\documentclass{amsart}

\usepackage[utf8]{inputenc}
\usepackage[numbers]{natbib} % or [authoryear]
\usepackage{hyperref}
\usepackage{float}
\usepackage{amsmath, amssymb, amsthm}
\usepackage{geometry}
\usepackage{graphicx}
\usepackage{mathrsfs} % For script fonts
\usepackage{tikz-cd}

\newtheorem{theorem}{Theorem}[section]

\newtheorem{lemma}[theorem]{Lemma}
\newtheorem{corollary}[theorem]{Corollary}
\theoremstyle{definition}
\newtheorem{definition}[theorem]{Definition}
\theoremstyle{remark}

\theoremstyle{proposition}
\newtheorem{proposition}[theorem]{Proposition}

\theoremstyle{definition}
\newtheorem{example}{Example}[section]
\title{Constructing Maximal Cohen-Macaulay Sheaves on Symplectic Singularities}
\author{Shang Xu}

\begin{document}
% Title Page

% Abstract
\begin{abstract}
In this paper, we study maximal Cohen-Macaulay sheaves on symplectic singularities. These sheaves generate the singularity categories and thus measure how far a singularity is from being smooth. We lift maximal Cohen-Macaulay sheaves on a singular variety to reflexive sheaves on its resolution and use Grothendieck duality to study their cohomological vanishing. We work this out in detail for the resolution $T^*\mathbb{P}^2 \rightarrow \mathcal{N}_{3,1}$, where $\mathcal{N}_{j,k}$ denotes the variety of nilpotent $j\times j$ matrices of rank at most $k$. In this case, we characterize the reflexive sheaves on $T^*\mathbb{P}^2$ whose pushforwards are maximal Cohen-Macaulay, and use vanishing results on $\mathbb{P}^2$ to construct many indecomposable maximal Cohen-Macaulay sheaves on $\mathcal{N}_{3,1}$. We also extend this construction to the resolution
$T^*\mathbb{P}^n \to \mathcal{N}_{n+1,1}$.
\end{abstract}

\maketitle

\tableofcontents
\section{Introduction}
Maximal Cohen-Macaulay sheaves play a central role in the study of singular algebraic varieties. They arise naturally as higher order syzygies of free resolutions and may be viewed as the sheaves that are as close as possible to vector bundles on a smooth space. In the Gorenstein case, they have full depth, admit no higher $\operatorname{Ext}$ against the structure sheaf, and form the generating objects of the singularity category (cf. \cite{Buchweitz1986},\cite{Orlov2004}). In this sense, they measure how far a singularity is from being smooth while still retaining enough rigidity to reflect subtle geometric and homological properties of the singularity. They also appear naturally in the study of non-commutative resolutions (\cite{Leuschke2007}, \cite{IyamaWemyss2014}), derived categories, and representation-theoretic approaches to birational geometry. 

On normal Cohen-Macaulay surfaces, reflexive sheaves are precisely maximal Cohen-Macaulay sheaves. Thus, the maximal Cohen-Macaulay condition can be regarded as a higher-dimensional analog of reflexivity. Indeed, on an $n$-dimensional Cohen-Macaulay scheme, maximal Cohen-Macaulay sheaves are exactly those satisfying Serre’s condition $S_n$, while on normal schemes, reflexive sheaves are exactly those satisfying $S_2$.

Symplectic singularities are those singular spaces whose smooth loci carry holomorphic symplectic forms extending across
resolutions in a controlled way. Grothendieck duality (cf. \cite{HartshorneRD}, III) for coherent sheaves provides a bridge between the homological properties of sheaves on these singular varieties and those on their resolutions, and this makes it possible to control the cohomology and Ext groups of sheaves on the resolution in order to construct maximal Cohen-Macaulay sheaves on the symplectic singularity.

The basic two-dimensional examples are the Kleinian, or ADE, surface singularities. In this case, a foundational result of Artin and Verdier shows that reflexive (hence maximal Cohen-Macaulay on surfaces) modules over rational double points are classified in terms of this ADE geometry, giving one of the earliest and most concrete manifestations of the McKay correspondence in algebraic geometry, cf. \cite{ArtinVerdierEsnault1985}, \cite{EsnaultKnorrer1985}. This also implies that the singular categories of rational double points are remarkably rigid and are governed by the corresponding ADE Dynkin diagrams. 

This surface case also admits a rich moduli theory. In particular, in \cite{Ishii1992} Ishii studied the stable reflexive sheaves and their deformations on those surfaces with only rational double points. It is proved there the general existence of reflexive sheaves with specific Chern class at each singularity, and that the singularities in the moduli space of reflexives have the same type as the surface itself. This reveals the importance of maximal Cohen-Macaulay sheaves in the sense that the singularities can be rebuilt through the moduli space of those sheaves. \cite{deBobadillaRomano2024} also studied the deformation and moduli of reflexive sheaves on normal Gorenstein Stein surfaces.

Motivated by this picture, it is natural to investigate maximal Cohen-Macaulay sheaves on higher-dimensional symplectic singularities and to ask how much of the ADE story survives: whether such sheaves can be described through geometry on a symplectic resolution, how vanishing and duality constrain them, and to what extent they admit explicit constructions and moduli-theoretic interpretations.

The aim of this paper is to study maximal Cohen-Macaulay sheaves on various symplectic singularities. The prototypes are Kleinian (ADE) surfaces , where all maximal Cohen-Macaulay sheaves are completely classified. Besides those prototypes, Springer resolutions (cf.\cite{ChrissGinzburg1997}) also form a rich family of symplectic resolutions, and the ADE singularities appear as singularities on slices (cf. \cite{Brieskorn1970}). While these are close to our prototypes since they are symplectic, Nakajima quiver varieties constructed in \cite{Nakajima1994Instantons} provide better generalizations, as one can choose Dynkin quivers to recover Kleinian surfaces. Therefore, our main interest is in how to construct maximal Cohen-Macaulay modules on those singular quiver varieties.

The paper is organized as follows. In Section \ref{S2}, we introduce the symplectic singularities that we are going to work on. In Section \ref{S3}, we establish some general properties of maximal Cohen-Macaulay sheaves on symplectic singularities. The main tool we use is the spectral sequence derived from Grothendieck duality in Lemma \ref{lem3.10}. In Section \ref{S4}, we apply our general theories to the resolution $T^*\mathbb{P}^2\to\mathcal{N}_{3,1}$ and characterize the maximal Cohen-Macaulay sheaves on $\mathcal{N}_{3,1}$ using reflexive sheaves on $T^*\mathbb{P}^2$. Here $\mathcal{N}_{j,k}$ denotes the variety of nilpotent $j\times j$ matrices of rank at most $k$. Specifically, we show that
\begin{proposition}[Proposition \ref{prop4.2}]\label{prop1.1}
Suppose $\mathcal{F}$ is a reflexive sheaf on $T^*\mathbb{P}^2$. Then $\pi_*\mathcal{F}$ is reflexive. It is maximal Cohen-Macaulay if and only if
\begin{enumerate}
\item $R^1\pi_*\mathcal{F}=0$.
\item $\text{Ext}^1_\mathcal{\widetilde{O}}(\mathcal{F},\mathcal{\widetilde{O}})=0$.
\item The map $ (R^2\pi_*\mathcal{F})'=\text{Ext}^4_\mathcal{O}(R^2\pi_*\mathcal{F},\mathcal{O})\rightarrow\text{Ext}^2_\mathcal{\widetilde{O}}(\mathcal{F},\mathcal{\widetilde{O}})$ is an isomorphism.
\end{enumerate}    
\end{proposition}

We are also interested in constructing concrete examples of maximal Cohen-Macaulay sheaves on those symplectic singular varieties. Proposition \ref{prop3.14} shows that a sheaf whose higher direct images, as well as those of its dual, vanish gives rise to a maximal Cohen-Macaulay sheaf on the singularity. We then use this criterion in Sections \ref{S4} and \ref{s5} to construct maximal Cohen-Macaulay sheaves on the subregular orbits or nilpotent cone. On our 4-dimensional example, we have found indecomposable maximal Cohen-Macaulay sheaves of all positive ranks.
\begin{theorem}[Corollary \ref{coro4.19}]
For any integer $r>0$, there exists an indecomposable maximal Cohen-Macaulay sheaf on $\mathcal{N}_{3,1}$.
\end{theorem}

We also identify certain ranks for which one can construct indecomposable maximal Cohen-Macaulay sheaves on $\mathcal{N}_{n+1,1}$.

\begin{theorem}[Corollary \ref{coro5.9}]
Suppose $n\geq 3$. For $r\geq n$, there is an integer $M_{r}$ such that for $k\geq M_{r}$, there exists an indecomposable maximal Cohen-Macaulay module on $\mathcal{N}_{n+1,1}$ of rank $kr$. If $n|r$, then we can set $M_{r}=1$.
\end{theorem}

\paragraph{\textbf{Conventions.}}
Throughout this paper, we work over the field $\mathbb{C}$. All varieties are assumed to be irreducible and quasi-projective unless otherwise stated. For a coherent sheaf $\mathcal{F}$ on a normal variety, we write $\mathcal{F}^{\vee}:=\mathcal{H}om(\mathcal{F},\mathcal{O}_X)$ for its dual. We use $(-)^{\vee\vee}$ to denote the reflexive hull. If $\pi\colon Y\to X$ is a morphism, then $R^i\pi_*$ denotes the $i$-th higher direct image functor.

\ 

\paragraph{\textbf{Acknowledgments.}} I would like to thank David Eisenbud for many valuable comments. Thanks also Peisheng Yu, Wenqing Wei, Jerry Yang, and Cameron Chang for helpful conversations.

\section{Symplectic Singularities, Nakajima Quiver varieties}\label{S2}
In this section, we explain the definitions and basic properties of the varieties and singularities that we are going to work on.
\subsection{Symplectic singularities}
We begin by introducing the definition of symplectic singularities. We use the definition in \cite{Beauville2000SymplecticSingularities}.
\begin{definition}[Symplectic singularity]
A variety $X$ has a \emph{symplectic singularity} at a point $p \in X$ if 
there exists an open neighborhood $U \subset X$ of $p$ such that:

\begin{enumerate}
    \item[(a)] $U$ is normal;
    
    \item[(b)] The smooth locus $U_{\mathrm{reg}}$ admits an algebraic
 symplectic $2$-form $\varphi$;
    
    \item[(c)] For any resolution of singularities 
    $f : \widetilde{U} \to U$, the pullback 
    $f^{*}\varphi$ defined on $f^{-1}(U_{\mathrm{reg}})$ extends to an
 algebraic $2$-form on $\widetilde{U}$.
\end{enumerate}
\end{definition}

We say $X$ has symplectic singularities if all its singular points are symplectic singularities, and we say $\pi:\widetilde{X}\rightarrow X$ is a symplectic resolution if $X$ has rational singularities and $\widetilde{X}$ is a (simultaneous) crepant resolution of $X$, which means $\widetilde{X}$ is smooth and the pullback of dualizing sheaf on $X$ is the dualizing sheaf on $\widetilde{X}$, namely $\pi^*\omega_X=\omega_{\widetilde{X}}$. Throughout this section, we suppose $X$ is an affine variety with symplectic singularities, and $\pi:\widetilde{X}\rightarrow X$ is a symplectic resolution.

A useful result in \cite{Beauville2000SymplecticSingularities} is:
\begin{lemma}\label{lem2.2}
    A symplectic singularity is rational Gorenstein.
\end{lemma}

Therefore, by the properties of Gorenstein rings, we see that $\mathcal{O}_{X}$ is the dualizing sheaf on $X$, and by the definition of rational singularities: $\mathbf{R}\pi_{*}\mathcal{O}_{\widetilde{{X}}}=\mathcal{O}_{X}$, or equivalently, $R^{>0}\pi_{*}\mathcal{O}_{\widetilde{X}}=0$, the natural map $\mathcal{O}_{X}\rightarrow\mathcal{O}_{\widetilde{X}}$ is an isomorphism. Since $X$ is supposed to be affine, this is also equivalent to saying $H^{>0}(\mathcal{O}_{\widetilde{X}})=0$ and $H^{0}(\mathcal{O}_{\widetilde{X}})=H^{0}(\mathcal{O}_{{X}})$.

\begin{example}[Kleinian surfaces as symplectic resolutions]
Suppose $X= \mathbb{C}^2/G$ where $G$ is a finite group in $SL(2,\mathbb{C})=Sp(2,\mathbb{C)}$. Then $X$ inherits a symplectic form from $\mathbb{C}^2$, and there is one isolated singular point on $X$ that is a symplectic singularity. These kinds of singularities are known as \emph{Rational double points}, \emph{ADE singularities} , or \emph{Kleinian singularities}. In \cite{ItoNakamura1996}, a canonical resolution is given as the $G$-Hilbert scheme: There is a projective morphism
$$
G\text{-Hilb}(\mathbb{C}^2)\rightarrow\mathbb{C}^2/G
$$
extending an isomorphism $G\text{-Hilb}((\mathbb{C}^2)^{reg})\rightarrow(\mathbb{C}^2)^{reg}/G$, which is a minimal resolution of symplectic singularities.
While early in \cite{Nakajima1994Instantons}, a symplectic resolution was given by Nakajima quiver varieties. These two resolutions are identified in \cite{CrawleyBoevey2001ExceptionalFibres}.
\end{example}
\begin{example}[Springer resolution as a symplectic resolution]
Let $G$ be a connected complex semisimple algebraic group with Lie algebra $\mathfrak{g}$, and let $\mathcal{N}\subset\mathfrak{g}$ be the nilpotent cone. Let $\mathcal{B}$ be the flag variety of $G$. The cotangent bundle $T^*\mathcal{B}$ carries its canonical symplectic form and admits a proper $G$-equivariant morphism
\[
\pi:T^*\mathcal{B}\to\mathcal{N},
\qquad
(x,\mathfrak{b})\mapsto x,
\]
where
\[
T^*\mathcal{B}=\{(x,\mathfrak{b})\in\mathfrak{g}\times\mathcal{B}\mid x\in\mathfrak{b},\; x\text{ nilpotent}\}.
\]
The map $\pi$ is proper, birational, and an isomorphism over the regular nilpotent orbit. Hence, $\pi$ is a resolution of singularities of $\mathcal{N}$. Moreover, the symplectic form on $T^*\mathcal{B}$ restricts to the Kostant--Kirillov form on smooth nilpotent orbits, so $\mathcal{N}$ has symplectic singularities in the sense of Beauville, and $\pi$ is a symplectic resolution.
\end{example}

\subsection{Nakajima Quiver Varieties}
Now we introduce an important family of symplectic resolutions. Those are called Nakajima quiver varieties. While this family includes almost all examples above (as we will see, the Kleinian surfaces and A type springer resolutions are Nakajima quiver varieties), it also provides other concrete examples. For example, the cycle map from the punctual Hilbert scheme of $\mathbb{C}^2$ to the symmetric product of $\mathbb{C}^2$: $\text{Hilb}^n(\mathbb{C}^{2})\rightarrow(\mathbb{C}^{2})^{(n)}$ is a symplectic resolution and can be realized as Nakajima quiver varieties. Our introduction basically follows \cite{Ginzburg2010Nakajima} but much briefer than that.

Fix a quiver $Q$ with vertex set $V={1,2,3,...,n}$ and edge set $E$. While fixing a dimension vector $v=(v_1,v_2,...,v_n)$ and assigning vertex i a dimension $v_i$, one can form a $\mathbb{C}$-vector space consisting of all representations of $Q$ with the action of an algebraic group of automorphisms at the vertices.
$$\text{Rep}(Q,v)=\prod_{e\in E}\text{Hom}(\mathbb{C}^{v_{h(e)}},\mathbb{C}^{v_{t(e)}}),\qquad GL(v)=\prod_{i=1}^{n}GL(v_i,\mathbb{C})
$$
Here $t(e)$ and $h(e)$ are the tail and head of the edge $e$. $GL(v)$ acts on $\text{Rep}(Q,v)$ by simultaneous conjugation
$$(g_1,g_2,...,g_n):(M_{e})_{e\in E}\mapsto (g_{h(e)}\cdot M_{e}\cdot g_{t(e)}^{-1})_{e\in E}
$$
Then you can form a quotient variety, either a categorical quotient or a GIT quotient with respect to some character $\theta$. Let's define:
$$\mathcal{M}^{\theta}(v)=\text{Rep}(Q,v)^{\theta\text{-ss}}/\!/\!_{\theta}\,GL(v),\quad\quad\mathcal{M}^{0}(v)=\text{Rep}(Q,v)/\!/GL(v)
$$
Here $\theta\text{-ss}$ refers to the semistable points with respect to the character $\theta$. These are called \emph{quiver varieties} of $Q$ with dimension vector $v$. There is a projective map $\mathcal{M}^{\theta}(v)\rightarrow\mathcal{M}^{0}(v)$, and under good stability conditions, this gives you a resolution of singularities.

Nakajima's idea in \cite{Nakajima1994Instantons} of constructing quiver varieties contains two additional steps. Firstly, we choose a framing on quiver $Q$; that is, we form a new quiver $Q^{\heartsuit}$ with vertex set $I\sqcup I'$, equipped with a bijection $ I\xrightarrow{\sim} I',\  i\mapsto i'$, and with the edge set equal to the disjoint union of edges in $Q$ connecting points in $I$ and a set of additional edges pointing from $i$ to $i'$ for each $i\in I$. Then, by fixing the dimension vector $w=(w_1,w_2,...,w_n)$ on those framing vertices, we can again take the $\mathbb{C}$ vector space $\text{Rep}(Q^{\heartsuit},v,w)$. While both $GL(v)$ and $GL(w)$ are acting on this space, we only care about the $GL(v)$-action when taking the quotient.

Instead of taking the quotient directly, we do \emph{Hamiltonian reduction} to $\text{Rep}(Q^{\heartsuit},v,w)$: Take its cotangent bundle $T^*\text{Rep}(Q^{\heartsuit},v,w)$ and identify its points by quadruples $(x,y,i,j)$, where $y,j$ are cotangent vectors. One can show that there is a $GL(v)$-equivariant momentum map $\mu: T^*\text{Rep}(Q^{\heartsuit},v,w)\rightarrow\mathfrak{gl}(v)^*$ defined by
$$\mu(x_e,x_{e^*},i_{k},j_{k})\mapsto\sum_{e\in E}(x_ex_{e^*}-x_{e^*}x_e)-\sum_{i=1}^{n}j_ki_k.
$$
Then fix any $GL(v)$-invariant character $\lambda\in(\mathfrak{gl}(v)^*)^{GL(v)}$. The fiber of $\lambda$ has a $GL(v)$ action, and one can again define the two kinds of quotients:
$$
\mathcal{M}^{\theta}_\lambda(v,w)=\mu^{-1}(0)^{\theta\text{-ss}}/\!/\!_{\theta}\,GL(v),\quad\quad \mathcal{M}^{0}_\lambda(v,w)=\mu^{-1}(0)/\!/GL(v)
$$
Again, there is a projective map $\mathcal{M}^{\theta}_\lambda(v,w)\rightarrow\mathcal{M}^{0}_\lambda(v,w)$. Those varieties play
the role of 'twisted cotangent bundles' on $X/G$, cf. \cite{ChrissGinzburg1997}. The importance of those varieties in this paper is that they form a rich family of symplectic resolutions, as shown in the following proposition, cf. \cite{Nakajima1994Instantons} and \cite{Nakajima1998QuiverKacMoody}.
\begin{proposition}
For fixed $v,w$ and a generic parameter $\theta$.
\begin{enumerate}
    \item $\mathcal{M}^{\theta}_\lambda(v,w)$ are smooth; they are also connected for generic $\lambda$.
    \item  For different generic parameters $(\lambda,\theta)$, $\mathcal{M}^{\theta}_\lambda(v,w)$ are diffeomorphic to each other.
    \item The map $\pi: \mathcal{M}^{\theta}_\lambda(v,w)\rightarrow\mathcal{M}^{0}_\lambda(v,w)$ is a symplectic resolution. 
\end{enumerate}
\end{proposition}
There is a more concrete interpretation as "$v$-regularity" for what is "generic" in \cite{Nakajima1994Instantons}, Theorem 2.8. However, in this paper, we are more interested in the concrete models for symplectic resolution generated by this method, so we omit the discussion here.

Before giving some examples, we provide several useful results about this resolution of Nakajima quiver varieties. From now on, we always suppose that the parameters $(\lambda,\theta)$ are generic.

Consider the $\mathbb{C}^\times$ action on $\text{Rep}(Q^{\heartsuit},v,w)$ contracting cotangent fibers. This action preserves $\mu^{-1}(0)$ and commutes with the $GL(v)$ action, so it induces an action on $\mathcal{M}^{\theta}_0(v,w)$, and the map $\pi: \mathcal{M}^{\theta}_0(v,w)\rightarrow\mathcal{M}^{0}_0(v,w)$ is $\mathbb{C}^\times$-equivariant. 

Define $\Lambda_\theta(v,w)=[\pi^{-1}(\mathcal{M}^{0}_0(v,w)^{\mathbb{C}^\times}]_{red}$. Then one has the following proposition:
\begin{proposition}
For a generic parameter $(0,\theta)$, we have:

\begin{enumerate}
\item Each irreducible component of the variety 
$\Lambda_{\theta}(v,w)$ is a Lagrangian subvariety of 
$\mathcal{M}_0^\theta(v,w)$. In particular, each component has dimension $\frac{1}{2}\dim\mathcal{M}_0^\theta(v,w)$.

\item Assume, in addition, that $\theta_i>0$ holds for all $1\leq i\leq n$ and the quiver $Q$ has no oriented cycles. then one has $M_{0,0}(v,w)^{\mathbb{C}^{\times}}=\{0\}$; thus, in this case, we have
$$
\Lambda_{\theta}(v,w)=\pi^{-1}(0).
$$
\end{enumerate}
\end{proposition}

Indeed, the $\mathbb{C}^\times$-action provides a contraction of the variety $\mathcal{M}_0^0(v,w)$
to $\mathcal{M}_0^0(v,w)^{\mathbb{C}^\times}$ and a contraction of the variety $\mathcal{M}_0^\theta(v,w)$
to $\mathcal{M}_0^\theta(v,w)^{\mathbb{C}^\times}$. Under the generic condition in the proposition, this gives a contraction of $\mathcal{M}_0^0(v,w)$ to the point $0$, and \cite{Nakajima1994Instantons} Corollary 5.5 tells you the following result:
\begin{proposition}
For generic $\theta$, $\mathcal{M}_0^\theta(v,w)$ is homotopy equivalent to $\pi^{-1}(0)$.
\end{proposition}

Now we include a stability criterion to help us concretely compute the GIT quotient. The following proposition is derived from results in \cite{King1994}, while one can also find an alternative proof in \cite{Nakajima1994Instantons}.

\begin{proposition}
A quadruple $(x,y,i,j)\in \mu^{-1}(\lambda)$ is $\theta$-semistable 
if and only if the following holds: For any collection of vector subspaces, 
$S=(S_i)_{1\le i\le n}\subset V=(V_i)_{1\le i\le n}$ 
which is stable under the maps $x$ and $y$, we have
\begin{align}\notag
S_i \subset \ker j_i,\ \forall\, 1\le i\le n
&\;\Longrightarrow\;
\theta\cdot(\dim S_i)_{i=1}^n\le 0, \\\notag
S_i \supset \operatorname{Im} i_i,\ \forall\, 1\le i\le n
&\;\Longrightarrow\;
\theta\cdot(\dim S_i)_{i=1}^n\le \theta\cdot(\dim S_i)_{i=1}^n. 
\end{align}
\end{proposition}
\begin{corollary}
    
In the special case where $\theta_i$ are all positive or negative, the point $(x,y,i,j)\in \mu^{-1}(\lambda)$ is $\theta$-semistable if and only if for any collection of vector subspaces $S=(S_i)_{1\le i\le n}\subset V=(V_i)_{1\le i\le n}$ that is stable under the maps $x$ and $y$, we have
\begin{enumerate}
\item If $\theta_i>0,\ \forall\  1\le i\le n$, then $S_i \subset \text{ker}\  j_i,\ \forall\, 1\le i\le n \;\Longrightarrow\;S=0$,
\item If $\theta_i<0,\ \forall\  1\le i\le n$, then $S_i \supset \text{Im} \ j_i,\ \forall\, 1\le i\le n \;\Longrightarrow\;S=0$,
\end{enumerate}
\end{corollary}

We will end this section by computing some examples.
\begin{example}[One vertex, no arrows]
Consider the quiver with one vertex and no arrows. Then we have $T^*\text{Rep}(Q^{\heartsuit},v,w)=\text{Hom}(V,W)\oplus \text{Hom}(W,V)$ acting by $\text{GL}(V)$.
The moment map is $\mu(i,j)=-ji$,
so $\mu^{-1}(\lambda)=\{(i,j)\mid ji=-\lambda\}$.

If $\lambda\neq 0$, then $\mu^{-1}(\lambda)\neq\varnothing$ if and only if $v\le w$, and the $G$-action is free. The map $\mu^{-1}(\lambda)\to \mathfrak{gl}(W),\  (i,j)\mapsto ij$ identifies $\mathcal M_\lambda^0(v)$ with the $\text{GL}(W)$-orbit of the diagonal matrix with $-\lambda$ repeated $v$ times and $0$ repeated $w-v$ times.

Now let $\lambda=0$. If $\theta>0$, then $(i,j)$ is $\theta$-semistable if and only if $i$ is injective. The condition $ji=0$ implies $j$ factors through $W/V$, hence $\mathcal M_0^\theta(v)=T^*\text{Gr}(v,W)$. If $\theta<0$, then $(i,j)$ is $\theta$-semistable if and only if $j$ is surjective. Then $i$ factors through $\ker j$, and up to $GL(V)$-conjugacy, $\mathcal M_0^\theta(v)=T^*\text{Gr}(w-v,W)$. If $\theta=0$, the image of $ \mu^{-1}(0)\to \mathfrak{gl}(W),\ (i,j)\mapsto ij $ consists of matrices with square zero and rank $\le \min(v,\lfloor w/2\rfloor)$, and $\mathcal M_0^0(v)$ equals this image.

One also derives
\end{example}
\begin{example}[One vertex, one loop]
Consider the quiver with one vertex and one loop. Assume $w=1$ and $\lambda=0$. Then we have
$T^*\text{Rep}(Q^{\heartsuit},v,1)=\text{End}(V)^{\oplus 2}\oplus V^*\oplus V$
acting by $\text{GL}(V)$. The moment map is
$\mu(X,Y,i,j)=[X,Y]-ji$.

First, consider $\theta=0$. There is a map
$(\mathbb C^{2v}) \to \mu^{-1}(0)$ given by
$$(x_1,\dots,x_v,y_1,\dots,y_v)\mapsto
(\text{diag}(x_i),\text{diag}(y_i),0,0)
$$
It descends to a morphism
$(\mathbb C^{2v})/S_v \to \mathcal M_0^0(v)$,
and this morphism is an isomorphism. Hence
$\mathcal M_0^0(v,1)\cong (\mathbb C^{2v})/ S_v=(\mathbb{C}^2)^{(v)}$.

Now consider $\theta<0$. $(X,Y,i,j)$ is $\theta$-semistable
if and only if $\mathbb C\langle X,Y\rangle i = V$.
If moreover $[X,Y]=ji$, one shows $j=0$, so $X$ and $Y$ commute.
This identifies $\mathcal M_0^{-}(v,1)$ with $\text{Hilb}^v(\mathbb C^2)$:
To the $GL(V)$-orbit of $(X,Y,i)$, we associate the ideal
$I=\{f\in\mathbb C[x,y]\mid f(X,Y)i=0\}$; Conversely, for a codimension $v$ ideal $I$ in $\mathbb{C}^2$, we set $V=\mathbb{C}^2/I$ and assign $i$ the map $\mathbb{C}\rightarrow V,\ 1\mapsto 1+I$ and take $X,Y$ to be the multiplication by $x$ and $y$.

Finally, for $\theta>0$, one obtains similarly
$\mathcal M_0^{+}(v)\cong \text{Hilb}^v(\mathbb C^2)$.
\end{example}

\begin{example}[ADE surfaces as quiver varieties]
Consider the affine Dynkin quivers of ADE type: $Q=\widetilde{A_n}, \widetilde{D_n}, \widetilde{E_6}, \widetilde{E_6}, \widetilde{E_8}$. Each quiver of those types can be realized as a McKay-Springer quiver of a finite group $G$ in $SL(2,\mathbb{C})$. We choose the dimension vector $V= (V_0,V_1,...,V_n)$ to be the set of dimensions of all irreducible representations of $G$ (Choosing $v_0=1$ corresponds to the trivial representation). Then, choosing $\theta_i$ to be all positive, it is shown in \cite{CrawleyBoevey2001ExceptionalFibres} that the Nakajima quiver variety $\mathcal M_0^{\theta}(v,0)$ is isomorphic to $G\text{-Hilb}(\mathbb{C}^2)$, and $\mathcal M_0^{0}(v,0)$ is isomorphic to $\mathbb{C}^2/G$, so this recovers the Kleinian resolutions.
\end{example}

\section{Maximal Cohen-Macaulay modules on Symplectic singularities}\label{S3}
The family of maximal Cohen-Macaulay modules is an important feature of singularities. Let's start from the definition: Suppose $R$ is a Cohen-Macaulay local ring with maximal ideal $\mathfrak{m}$. An $R$ module $M$ is called Cohen-Macaulay if $\dim(\text{Supp}(M))=\text{depth} M$. It is called maximal Cohen-Macaulay if both of those two numbers equal $\dim R$. By the cohomological interpretation of depth (cf.\cite{Hartshorne1977} III. \S3, Exercise 3.4), this is also equivalent to $H_\mathfrak{m}^i(M)=0$ for $i< \dim R$. We call a sheaf $\mathcal{F}$ maximal Cohen-Macaulay if $\mathcal{F}_p$ for all points $p$ (including those non-closed ones).

Maximal Cohen-Macaulay modules arise naturally in syzygies of free resolutions. Again, suppose $R$ is a local ring and $M$ is an $R$ module. Take generators of $M$ so we can resolve $M$ using a free module $F$ that maps onto $M$. Then you get a short exact sequence $0\rightarrow K\rightarrow F\rightarrow M\rightarrow 0$. Taking local cohomology, we see that if $\text{depth}M<\text{depth F}=\dim R$, then there is $\text{depth}K=\text{depth} M+1$, and if $M$ is maximal Cohen-Macaulay, then so is $K$. Hence, if you take a free resolution of a module, the depths of syzygies become larger and larger until they reach $\dim R$, and after this, all syzygies are maximal Cohen-Macaulay. Since we are going to use this many times, we will record this as a lemma.
\begin{lemma}\label{lem3.1}
Suppose $R$ is a local ring and $M$ is an $R$-module. Take a free resolution $\cdot\cdot\cdot\xrightarrow{d_2}R^{n_1}\xrightarrow{d_1}R^{n_0}\xrightarrow{d_0}M\rightarrow 0$ and denote the syzygies by \,$N_i=\mathrm{im}\, d_i$; then 
$$\mathrm{depth}(N_i)\geq\min\{\mathrm{depth}R, \mathrm{depth}M+i\}$$
\end{lemma}
\begin{proof}
Repeatedly use \cite{Eisenbud1995} Corollary 18.6.
\end{proof}
For regular local rings, all modules have finite projective dimension. Hence, by the Auslander–Buchsbaum theorem (cf.\cite{Eisenbud1995}), maximal Cohen-Macaulay modules are the same as projective modules, and this occurs only in regular local rings. This means that for any singular local ring $R$, there is a subcategory in $R$ modules consisting of maximal Cohen-Macaulay modules that are not projective.

In derived categories, there is a clear result explaining that maximal Cohen-Macaulay modules are essential features of Gorenstein singularities. For variety $X$, we define the \emph{singular category} $D\!_{sing}(X)$ by the Verdier quotient $D^b(\text{Coh}(X))/\text{Perf}(X)$, where $\text{Perf}(X)$ denotes the subcategory of perfect complexes in $D^b(\text{Coh}(X))$. Here, a complex of coherent sheaves is said to be perfect if it is quasi-isomorphic to a bounded complex of locally free sheaves of finite rank. If $X$ is smooth, again by the Auslander–Buchsbaum theorem, any coherent sheaf admits a finite length locally free resolution , so the singular category is trivial, and this only happens when $X$ is nonsingular, explaining this terminology. Now, the following proposition tells us about the generators in the singular category of Gorenstein varieties, cf.\cite{Buchweitz1986}, \cite{Orlov2004}, and \cite{Lu2020TriangleEquivalences}.

\begin{proposition}
Let $X$ be a quasi-projective Gorenstein variety, then every object in $D\!_{sing}(X)$ is isomorphic to a maximal Cohen-Macaulay sheaf concentrated at zero degree. Furthermore, $D\!_{sing}(X)$ is isomorphic to the stabilized category $\underline{MCM}(X)$, whose objects are maximal Cohen-Macaulay sheaves and morphisms are morphisms between coherent sheaves quotient by those factoring through vector bundles.
\end{proposition}

The maximal Cohen-Macaulay condition is actually homological: Applying Grothendieck local duality (cf.\cite{HartshorneRD}, V), we obtain the following lemma:

\begin{lemma}\label{lem3.3}
A coherent sheaf $\mathcal{F}$ on a Gorenstein variety $X$ is maximal Cohen-Macaulay if and only if $\mathcal{E}xt^{i}(\mathcal{F}, \mathcal{O}_X)=0$ for all $i>0$.
\end{lemma}

Now let's consider maximal Cohen-Macaulay sheaves on symplectic singularities. Suppose $X=\text{Spec} \mathcal{O}$ is an affine variety with only symplectic singularities, and $\pi:\widetilde{X}\rightarrow X$ is a symplectic resolution. We denote the structure sheaf of $\widetilde{X}$ by $\widetilde{\mathcal{O}}$. Now consider a finitely generated $\mathcal{O}$-module $M$ (equivalently, a coherent sheaf on $X$, so we abuse the notation of sheaves and modules here).  Define 
$$
\widetilde{M}=\pi^*M/\text{Tor}(\pi^*M),\quad \overline{M}=(\pi^*M)^{\vee\vee}
$$
where $\text{Tor}(\pi^*M)$ is the torsion part of $\pi^*M$ and $(\pi^*M)^{\vee\vee}$ is taking the double dual. One can see that $\widetilde{M}$ is torsion-free and $\overline{M}$ is reflexive, and there is a canonical injection $\widetilde{M}\hookrightarrow\overline{M}$. Actually, there is more we can say with this injection, as the following lemma shows. A module $M$ (a sheaf $\mathcal{F}$, respectively) is called \emph{reflexive} if it is finitely generated (coherent, respectively) and the canonical map $M\rightarrow M^{\vee\vee}$ ($\mathcal{F}\rightarrow \mathcal{F}^{\vee\vee}$, respectively) is an isomorphism, we have
\begin{lemma}
If $M$ is a reflexive $\mathcal{O}$-module, the canonical injection $\widetilde{M}\hookrightarrow\overline{M}$ induces isomorphisms $\pi_*\widetilde{M}=\pi_*\overline{M}=M$.
\end{lemma}
\begin{proof}
Since pushforward is left exact, we see that the induced map $\pi_*\widetilde{M}\rightarrow\pi_*\overline{M}$ is injective. To prove surjectivity, we consider the natural map $\pi^*M^\vee\rightarrow(\pi^*M)^\vee$. Taking the dual, we get a map $(\pi^*M)^{\vee\vee}\rightarrow(\pi^*M^\vee)^{\vee}$. This map is injective: Since $\pi$ is birational, one sees that this map is generically injective, so its kernel is torsion. Since the dual of a coherent sheaf is reflexive, we see that, in particular, $(\pi^*M)^{\vee\vee}$ is torsion free; hence, the kernel must be trivial. Now taking pushforward, one gets a map $\pi_*(\pi^*M)^{\vee\vee}\hookrightarrow\pi_*(\pi^*M^\vee)^{\vee}$. By the adjunction of pushforward and pullback, and Lemma \ref{lem2.2}, we see that:
\begin{align}\notag
\pi_*(\pi^*M^\vee)^{\vee}
&=\pi_*\mathcal{H}om_{\widetilde{\mathcal{O}}}(\pi^*M^\vee,\widetilde{\mathcal{O}})\\\notag
&=\mathcal{H}om_{\mathcal{O}}(M^\vee,\pi_*\widetilde{\mathcal{O}})\\\notag
&=\mathcal{H}om_{\mathcal{O}}(M^\vee,\mathcal{O})\\\notag
&=M
\end{align}

By naturality, we have a commutative diagram
$$
\begin{tikzcd}
& \pi_*\widetilde{M} \arrow[r, hook]\arrow[rd,dashed,hook] & \pi_*\overline{M} \arrow[d,hook]\\
& M \arrow[u] \arrow[r, "\sim"] & \pi_*(\pi^*M^\vee)^{\vee}=M
\end{tikzcd}
$$
where the left vertical arrow comes from the natural map 
$$M\rightarrow \pi_*\pi^*M\rightarrow\pi_*(\pi^*M/\text{Tor}(\pi^*M))=\widetilde{M}
$$ 
Then we observe that any element in $\pi_*\overline{M}$ has an image in $M$ via the right vertical arrow, and this image can be traced back to the left bottom $M$ and then mapped up to $\pi_*\widetilde{M}$, giving an inverse image for this element. Hence, the map is surjective and thus an isomorphism. To prove they both equal $M$, it suffices to notice that the dashed arrow gives a splitting of $M$ into $\pi_*\widetilde{M}$, while the injectivity of this splitting implies that it is an isomorphism.
\end{proof}

On surface singularities, in particular, we have
\begin{corollary}
If $\dim X=2$, then $\widetilde{M}=\overline{M}$.
\end{corollary}
\begin{proof}
Consider the short exact sequence $ 0\rightarrow\widetilde{M}\rightarrow\overline{M}\rightarrow Q\rightarrow 0$. Taking the long exact sequence, we get
$$ 0\rightarrow H^0(\widetilde{M})\rightarrow H^0(\overline{M})\rightarrow H^0(Q)\rightarrow H^1(\widetilde{M})
$$
By Lemma \ref{lem3.3}, since we suppose $X$ is affine, we see in particular that pushing forward is the same as taking global sections,so the map $H^0(\widetilde{M})\rightarrow H^0(\overline{M})$ is an isomorphism. Since $M$ is finitely generated, there is a free module mapping onto $M$, say $\mathcal{O}^m\twoheadrightarrow M$. Then one pulls this back to get $\mathcal{\widetilde{O}}^m\twoheadrightarrow \pi^*M\twoheadrightarrow \pi^*M/\text{Tor}(\pi^*M)=\widetilde{M}$. 

Now, since $\dim X= 2$, the relative dimension of $\pi$ is at most 1; hence, by Corollary 11.2 in \cite{Hartshorne1977}, III, we have $H^2(\mathcal{F})=R^2\pi_*\mathcal{F}=0$ for any coherent sheaf $\mathcal{F}$. Applying the long exact sequence, we see that $H^1(\mathcal{\widetilde{O}}^m)\rightarrow H^1(\widetilde{M})$ is surjective. By Lemma \ref{lem2.2}, we have $H^1(\mathcal{\widetilde{O}})=0$, hence $H^1(\widetilde{M})=0$. This implies $H^0(Q)=0$. While $\widetilde{M}$ and $\overline{M}$ coincide on codimension 1 points, we see that $Q$ is supported on closed points. Therefore $H^0(Q)=0$ implies $Q=0$, and hence $\widetilde{M}=\overline{M}$.
\end{proof}

This observation is the starting point in \cite{ArtinVerdierEsnault1985}: If you pull back a reflexive sheaf $M$ on a rational double point to its resolution and quotient by its torsion, you get a reflexive sheaf $\widetilde{M}$, hence a vector bundle on the resolution. Classifying those vector bundles gives you the classification of reflexive (hence maximal Cohen-Macaulay on surfaces) sheaves on ADE singularities.
\begin{theorem}[\cite{ArtinVerdierEsnault1985}, 1.11]
Let $X_P=\text{Spec}\,\mathcal{O}_P$ be a rational double point, then $M\mapsto c_1(\widetilde{M})$ gives a one-to-one correspondence between the isomorphic classes of indecomposable reflexive modules on $X_P$ and the reduced transversal divisors to each component of exceptional fiber. 
\end{theorem}

On higher dimensional singularity, one may ask whether $\widetilde{M}=\overline{M}$ still holds. Now the following lemma characterizes their difference. Denote by $X_{\text{sing}}$ the singular locus of $X$.

\begin{lemma}\label{lem3.7}
Suppose $\mathrm{codim}_{\widetilde{X}}\pi^{-1}(X_{\text{sing}})\geq 2$. Given a maximal Cohen-Macaulay sheaf $M$ on $X$, we take a free representation of $M^\vee$, $\mathcal{O}^{n_1}\xrightarrow{\phi_1}\mathcal{O}^{n_0}\xrightarrow{\phi_0} M^\vee\rightarrow 0$, and define $N_1=\mathrm{ker}\phi_1$ and $N_2=\mathrm{ker}\phi_2$, we have
\begin{enumerate}
\item $N_1,N_2$ are maximal Cohen-Macaulay modules.
\item $\overline{M}/\widetilde{M}=\mathrm{Tor}(\pi^*N_1^\vee)=\mathcal{T}\!or^\mathcal{O}_1(N_2,\widetilde{\mathcal{O}})$.
\end{enumerate}
\end{lemma}
\begin{proof}
Firstly, take a free resolution of $M$, say $F_\bullet\rightarrow M\rightarrow 0$. Since $M$ is maximal Cohen-Macaulay, by Lemma \ref{lem3.3} one has $\text{Ext}_\mathcal{O}^i(M,\mathcal{O})=0$ for all $i>0$. Hence, taking the dual of this complex still yields an acyclic complex $0\rightarrow M^\vee\rightarrow F^\vee_\bullet$. Then by Lemma \ref{lem3.1}, one sees that $M^\vee$ is also maximal Cohen-Macaulay since it is a high order syzygy (let's say, the $i>\dim \mathcal{O}$-th syzygy, since you can truncate the resolution anywhere) of some module. Then (1) follows again from Lemma \ref{lem3.1}.

To show (2), consider the short exact sequence
$$
0\rightarrow N_0\rightarrow \mathcal{O}^{n_0}\rightarrow M^\vee\rightarrow 0
$$
Since $N_0$ is maximal Cohen-Macaulay, $\text{Ext}_\mathcal{O}^1(N_0,\mathcal{O})=0$,so taking the dual still gives you a short exact sequence
$$
0\rightarrow M \rightarrow \mathcal{O}^{n_0}\rightarrow N_0^\vee\rightarrow 0
$$
Applying pullback, since $\pi$ is birational, $\pi^*M\rightarrow\widetilde{\mathcal{O}}^{n_0}$ is generically injective, and since $\widetilde{\mathcal{O}}^{n_0}$ is free, the kernel coincides with $\text{Tor}(\pi^*M)$; hence, one gets
$$
0\rightarrow \widetilde{M} \rightarrow \widetilde{\mathcal{O}}^{n_0}\rightarrow \pi^*N_0^\vee\rightarrow 0
$$
Suppose $K$ is the kernel of $\widetilde{\mathcal{O}}^{n_0}\twoheadrightarrow\pi^*N_0^\vee\twoheadrightarrow\widetilde{N_0^\vee}$. Then, taking local cohomology, one sees that $K$ is reflexive. By taking the double dual of the map $\pi^*M\rightarrow K$, we get an injection $\overline{M}\hookrightarrow K$ which is also surjective on nonsingular locus. By assumption $\mathrm{codim}_{\widetilde{X}}\pi^{-1}(X_{\text{sing}})\geq 2$, we see $\overline{M}=K$ because two reflexive sheaves coinciding on an open set except a codimension $\geq 2$ locus are the same. So
there is a short exact sequence
$$
0\rightarrow \overline{M} \rightarrow \widetilde{\mathcal{O}}^{n_0}\rightarrow \widetilde{N_0^\vee}\rightarrow 0
$$
Comparing the two short exact sequences above, one gets the first relation in (2). The second follows from the definition of $\mathcal{T}\!or$ and the exact sequence $0\rightarrow M\rightarrow\mathcal{O}^{n_0}\rightarrow\mathcal{O}^{n_1}\rightarrow N_1^\vee\rightarrow 0$, which is the dual of the original one.
\end{proof}

Using this, one can easily construct a module $M$ with non-vanishing $\overline{M}/\widetilde{M}$ on those resolutions with exceptional locus of codimension $\geq 2$. To do this, we need the following lemma.
\begin{lemma}\label{lem3.8}
For any singular local ring $R$ and regular local ring $S$ equipped with a local homomorphism $R\rightarrow S$ inducing an isomorphism on residue fields, we have $pd_{R}S=\infty$.
\end{lemma}
\begin{proof}
Suppose conversely $pd_{R}S<\infty$, then by the General Change of Rings Theorem (cf.\cite{Weibel1994} Theorem 4.3.1), for any $S$-module $A$, we have $pd_R(A)\leq pd_R(S)+pd_S(A)\leq pd_R(S)+\dim S<\infty$. Now take $A$ to be the residue field $k$ of $S$, then as an $R$-module $k$ is also the residue field of $R$ by assumption. However, since $R$ is singular, we should have $pd_R(k)=\infty$, which is a contradiction.
\end{proof}
\begin{proposition}
Suppose\, $\mathrm{codim}_{\widetilde{X}}\pi^{-1}(X_{\text{sing}})\geq 2$, there exists a maximal Cohen-Macaulay module $M$ such that $\overline{M}/\widetilde{M}\neq 0$. 
\end{proposition}
\begin{proof}
Suppose $\dim X=\dim \widetilde{X}=d$. Choose a singular point $P$ on $X$ and a point $Q\in\pi^{-1}(P)$ and denote the local rings by $\mathcal{O}_P$ and $\mathcal{\widetilde{O}}_Q$. By Lemma \ref{lem3.8}, we have $pd_{\mathcal{O}_P}\mathcal{\widetilde{O}}_Q=\infty$, so there exists a $\mathcal{O}_P$-module $M_P$ such that $\text{Tor}^{\mathcal{O}_P}_{d+1}(N_P,\mathcal{\widetilde{O}}_Q)\neq 0$. Now construct a finitely generated $\mathcal{O}$-module $N$ whose stalk at $P$ is $N_P$ (one can easily do this by lifting a finite presentation of $N_P$), and take a d+1 step free resolution of $N$
$$0\rightarrow K\xrightarrow{\phi_d} \mathcal{O}^{n_d}\xrightarrow{\phi_{d-1}}\cdot\cdot\cdot\rightarrow\mathcal{O}^{n_1}\xrightarrow{\phi_1}\mathcal{O}^{n_0}\xrightarrow{\phi_0} N\rightarrow 0.
$$
where $K$ is the kernel of $\phi_d$. Then, by Lemma \ref{lem3.1}, $K$ is maximal Cohen-Macaulay. Since $\text{Tor}^{\mathcal{O}_P}_{d+1}(N_P,\mathcal{\widetilde{O}}_Q)\neq 0$, we have $\mathcal{T}\!or_{d+1}^\mathcal{O}(N,\mathcal{\widetilde{O}})\neq 0$; hence, the canonical argument in homological algebra gives you $\mathcal{T}\!or_{1}^\mathcal{O}(K,\mathcal{\widetilde{O}})=\mathcal{T}\!or_{d+1}^\mathcal{O}(N,\mathcal{\widetilde{O}})\neq 0$. Now take a free representation of $K$ and suppose the second syzygy is $M$, i.e. one has an exact sequence
$$0\rightarrow M\rightarrow\mathcal{O}^a\rightarrow\mathcal{O}^b\rightarrow K\rightarrow 0
$$
Then $M$ is maximal Cohen-Macaulay by Lemma \ref{lem3.1}, and by Lemma \ref{lem3.7}, $\overline{M}/\widetilde{M}=\mathcal{T}\!or_{1}^\mathcal{O}(K,\mathcal{\widetilde{O}})\neq 0$.
\end{proof}

Since $\widetilde{M}$ and $\overline{M}$ are different, it's a natural question to ask which one works better when we try to do some work here. The first observation is that if the relative dimension is $n$, then $R^n\pi_*\widetilde{M}=0$ because $\widetilde{M}$ is generated by global sections and $R^n\pi_*\widetilde{\mathcal{O}}=0$, but this does not necessarily hold for $\overline{M}$. However, the advantages of considering $\overline{M}$ are more subtle. Let's begin with the key lemma that helps us consider the homological relations between $X$ and $\widetilde{X}$.
\begin{lemma}\label{lem3.10}
For any coherent sheaf $\mathcal{F}$ on $\widetilde{X}$, there is a converging spectral sequence
$$E_2^{i,j}=Ext_{\mathcal{\widetilde{O}}}^i(R^{-j}\pi_*\mathcal{F},\mathcal{\widetilde{O}})\Rightarrow Ext_{\mathcal{O}}^{i+j}(\mathcal{F}
,\mathcal{O})
$$
\end{lemma}
\begin{proof}
Suppose $\dim X=\dim \widetilde{X}=n$. By Grothendieck duality, we have an identity in the derived category
$$\mathbf{R}\mathcal{H}om_X(\mathbf{R}\pi_* \mathcal{F},\mathcal{O}[n])=\mathbf{R}\pi_*\mathbf{R}\mathcal{H}om_{\widetilde{X}}(\mathcal{F},\pi^!(\mathcal{O}[n]))
$$
Now $X$ is affine, so the left hand side is essentially $\mathbf{R}\text{Hom}_\mathcal{O}(\mathbf{R}\pi_* \mathcal{F},\mathcal{O}[n])$. Since $X$ is Gorenstein, $\mathcal{O}[n]$ is the dualizing complex on $X$. One can also see that $\mathcal{\widetilde{O}}[n]$ is the dualizing complex on $\widetilde{X}$ because $\widetilde{X}$ is smooth symplectic. Now the upper shriek $\pi^!$ pulls the dualizing complex on $X$ to the dualizing complex on $\widetilde{X}$, so $\pi^!(\mathcal{O}[n])=\mathcal{\widetilde{O}}[n]$.
Then we see the right hand side is $\mathbf{R}\pi_*\mathbf{R}\mathcal{H}om_{\widetilde{X}}(\mathcal{F},\widetilde{\mathcal{O}}[n])$, which is quasi-isomorphic to $\mathbf{R}\text{Hom}_\mathcal{\widetilde{O}}(\mathcal{F},\mathcal{\widetilde{O}}[n])$ since $X$ is affine. Finally, by twisting both sides by $[-n]$ and taking the cohomology, the Grothendieck-Serre spectral sequence gives you the result.
\end{proof}

Before introducing properties of $\overline{M}$, we need two lemmas.

\begin{lemma}\label{lem3.11}
Suppose $\mathcal{F}$ is a reflexive sheaf on a $n$-dimensional quasi-projective nonsingular variety $Y$. If the cohomological dimension of $Y$ is less than $n-1$, that is, $H^{n-1}(\mathcal{G})=0$ for any sheaf $\mathcal{G}$ on $Y$. Then $Ext_{\mathcal{O}_Y}^n(\mathcal{F},\mathcal{O}_Y)=Ext_{\mathcal{O}_Y}^{n-1}(\mathcal{F},\mathcal{O}_Y)=0$.
\end{lemma}
\begin{proof}
Firstly, we have a Grothendieck-Serre spectral sequence
$$
E_2^{i,j}=H^{i}(\mathcal{E}xt^j(\mathcal{F},\mathcal{O}_Y))\Rightarrow \text{Ext}^{i+j}(\mathcal{F},\mathcal{O}_Y)
$$
Suppose $P$ is the generic point of a subvariety of codimension $h_P$. The reflexive condition is the same as Serre's condition $S_2$, i.e., $\text{depth}\mathcal{F}_P\geq\min\{2,h_P\}$. Since $Y$ is nonsingular, the Auslander–Buchsbaum theorem implies $\mathcal{F}$ is locally free at those points with $h_P\leq 2$ and has a length $h_P-2$ projective resolution at those points with $h_P> 2$. Then one sees that $\text{Ext}_{\mathcal{O}_P}^j(\mathcal{F},\mathcal{O}_P)=0$
for $j\geq \max \{h_P-1,1\}$, hence $\dim\text{Supp}(\mathcal{E}xt^j(\mathcal{F},\mathcal{O}_Y))\leq n-j-2$ for $j\geq 1$. This implies $H^{n-j-1}(\mathcal{E}xt^j(\mathcal{F},\mathcal{O}_Y))=H^{n-j}(\mathcal{E}xt^j(\mathcal{F},\mathcal{O}_Y))=0$ for $j\geq 1$. Then the spectral sequence tells us $Ext_{\mathcal{O}_Y}^n(\mathcal{F},\mathcal{O}_Y)=H^n(\mathcal{F}^\vee)$ and $Ext_{\mathcal{O}_Y}^{n-1}(\mathcal{F},\mathcal{O}_Y)=H^{n-1}(\mathcal{F}^\vee)$. Finally, since the cohomology dimension of $Y$ is less than $n-1$, those two are all trivial.
\end{proof}
\begin{lemma}\label{lem3.12}
Suppose the fibers of $\pi$ at closed points have dimension $\leq n-2$. For a reflexive sheaf $\mathcal{F}$ on $X$, if $\pi_*\mathcal{F}$ is maximal Cohen-Macaulay, we have $\text{Ext}_{\mathcal{O}}^{n}(R^1\pi_*\mathcal{F},\mathcal{O})=0$.
\end{lemma}
\begin{proof}
In the spectral sequence of Lemma \ref{lem3.10}, the only nontrivial differential related to $\text{Ext}_{\mathcal{O}}^{n}(R^1\pi_*\mathcal{F},\mathcal{O})$ is $d_2^{n-2,0}$ with source $\text{Ext}_{\mathcal{O}}^{n-2}(\pi_*\mathcal{F},\mathcal{O})$. Since $\pi_*\mathcal{F}$ is reflexive, local duality gives us $\text{Ext}_{\mathcal{O}}^{n-2}(\pi_*\mathcal{F},\mathcal{O})=0$. Then one can see that in the spectral sequence, all differentials related to $\text{Ext}_{\mathcal{O}}^{n}(R^1\pi_*\mathcal{F},\mathcal{O})$ vanish, so the convergence gives us a natural injection $\text{Ext}_{\mathcal{O}}^{n}(R^1\pi_*\mathcal{F},\mathcal{O})\hookrightarrow\text{Ext}_{\mathcal{\widetilde{O}}}^{n-1}(\mathcal{F},\mathcal{\widetilde{O}})$. Under the assumption  that the fibers of $\pi$ at closed points have dimension $\leq n-2$, we see that the relative dimension of $\pi$ is less than $n-1$, so $H^{n-1}(\mathcal{G})=R^{n-1}\pi_*\mathcal{G}=0$ for any sheaf $\mathcal{G}$. Then Lemma \ref{lem3.11} gives us $\text{Ext}_{\mathcal{\widetilde{O}}}^{n-1}(\mathcal{F},\mathcal{\widetilde{O}})=0$, hence its sub $\text{Ext}_{\mathcal{O}}^{n}(R^1\pi_*\mathcal{F},\mathcal{O})=0$.
\end{proof}

Now the following proposition characterizes $\overline{M}$.
\begin{proposition}\label{prop3.13}
Suppose $\mathrm{codim}_{\widetilde{X}}\pi^{-1}(X_{\text{sing}})\geq 2$ and the fibers of $\pi$ at closed points have dimension $\leq n-2$.  For a maximal Cohen-Macaulay sheaf $M$ on $X$, we have
\begin{enumerate}
\item For any nonsingular point or isolated singularity $P$ on $X$, $(R^1\pi_*\overline{M})_P=0$.
\item If there is a reflexive sheaf $\mathcal{F}$ such that $\pi_*\mathcal{F}=M$, then $\mathcal{F}=\overline{M}$.
\item $\overline{M}^\vee=\overline{M^\vee}$, hence in particular, $\pi_*(\overline{M}^{\vee})=M^\vee$.
\end{enumerate}
\end{proposition}
\begin{proof}
By lemma \ref{lem3.12}, $\text{Ext}_{\mathcal{O}}^{n}(R^1\pi_*\overline{M},\mathcal{O})=0$. By Grothendieck local duality, this means $R^1\pi_*\overline{M}$ does not have sections supported on closed points (since on closed points $P$ we have $H_{\mathfrak{m}_P}^0(R^1\pi_*\overline{M})$ dual to $\text{Ext}_{\mathcal{O}_P}^{n}((R^1\pi_*\overline{M})_P,\mathcal{O}_P)=0$). Since $R^1\pi_*\overline{M}$ vanishes on the smooth locus, we see that it must also vanish on those isolated singularities, so (1) follows.

Now we prove (2). Suppose $\mathcal{F}$ is a reflexive sheaf with $\pi_*\mathcal{F}=M$. Then the isomorphism $M\xrightarrow{\sim}\pi_*\mathcal{F}$ induces a morphism $\pi^*M\rightarrow\mathcal{F}$, the double dual of which gives you a morphism $\overline{M}\rightarrow\mathcal{F}$. By assumption, the exceptional locus has codimension $\geq 2$, so this morphism is an isomorphism except for a codimension $\geq 2$ locus. Since both sheaves are reflexive, this morphism must be an isomorphism. Hence, we have $\mathcal{F}=\overline{M}$. 

The proof of (3) is similar: One has a natural map $\pi^*M^\vee\rightarrow(\pi^*M)^\vee$, and the double dual of this gives $\overline{M^\vee}\rightarrow(\pi^*M)^\vee=\overline{M}^\vee$, which is an isomorphism except for a codimension $\geq 2$ locus. Since both sheaves are reflexive, this is an isomorphism.
\end{proof}
We see that under the assumptions of this proposition, the $M\mapsto\overline{M}$ provides a one to one correspondence between reflexive sheaves on $X$ and the reflexive sheaves on $\widetilde{X}$ whose pushforward is reflexive, and this correspondence is preserved under duality. Furthermore, if $X$ has only isolated singularities, then $R^1\pi_*\overline{M}=0$.

The following proposition is useful for constructing maximal Cohen-Macaulay modules.
\begin{proposition}\label{prop3.14}
Suppose $\mathcal{F}$ is a vector bundle on $\widetilde{X}$ such that
$$R^{>0}\pi_*\mathcal{F}=R^{>0}\pi_*\mathcal{F}^\vee=0$$
then $M=\pi_*\mathcal{F}$ is maximal Cohen-Macaulay.
\end{proposition}
\begin{proof}
If $R^{>0}\pi_*\mathcal{F}$, we see that the spectral sequence in Lemma \ref{lem3.10} collapses to the $0^{th}$ column, so it gives you $\text{Ext}_{\mathcal{O}}^i(M,\mathcal{O})=\text{Ext}_{\mathcal{O}}^i(\pi_*\mathcal{F},\mathcal{O})=\text{Ext}^i_\mathcal{\widetilde{O}}(\mathcal{F},\mathcal{\widetilde{O}})=H^i(\mathcal{F}^\vee)=0$. Since $X$ is affine, by Lemma \ref{lem3.3} we see that $M$ is maximal Cohen-Macaulay.
\end{proof}

\section{Constructions on $T^*\mathbb{P}^2\rightarrow \mathcal{N}_{3,1}$}\label{S4}
We have seen many general constructions of reflexive/maximal Cohen-Macaulay sheaves on symplectic singularities in the previous section. Now we try to apply our general theory to specific resolutions. As the 2-dimensional case, the ADE singularities, has been fully explained in \cite{ArtinVerdierEsnault1985}, we apply our general theory to the simplest higher dimensional singularity: the subregular Springer resolution $T^*\mathbb{P}^2\rightarrow \mathcal{N}_{3,1}$, where $\mathcal{N}_{3,1}$ is the variety parameterizing $3\times 3$ nilpotent matrices with rank $\leq 1$.
\subsection{Characterization of maximal Cohen-Macaulay condition}
Throughout this section, we continue our tradition of denoting by $\pi$ the resolution $T^*\mathbb{P}^2\rightarrow \mathcal{N}_{3,1}$, and by $\mathcal{O}$ and $\widetilde{\mathcal{O}}$ the structure sheaves of $T^*\mathbb{P}^2$ and $\mathcal{N}_{3,1}$, respectively. Since the conjugation action of $GL(3,\mathbb{C})$ is transitive on those nilpotent $3\times 3$ matrices with rank 1, by generic smoothness, we see that the zero matrix, which we always denote by $0$, is the only singularity on $\mathcal{N}_{3,1}$, and the exceptional fiber is exactly the zero section of $T^*\mathbb{P}^2$.

To consider maximal Cohen-Macaulay modules, we first need a lemma
\begin{lemma}
Suppose $\mathcal{F}$ is a reflexive sheaf on $T^*\mathbb{P}^2$. The spectral sequence in Lemma \ref{lem3.10} gives a map $(R^2\pi_*\mathcal{F})'\rightarrow\text{Ext}^2_\mathcal{\widetilde{O}}(\mathcal{F},\mathcal{\widetilde{O}})$.
\end{lemma}
\begin{proof}
First, Since the map $T^*\mathbb{P}^2\rightarrow \mathcal{N}_{3,1}=\text{Spec}\,\mathcal{O}$ is projective and $\mathcal{F}$ is coherent, we see that $R^2\pi_*\mathcal{F}=H^2(\mathcal{F})$ is a finitely generated $\mathcal{O}$-module by \cite{Hartshorne1977} III, Theorem 5.2. One can also see that they are supported only at the singular point 0, so it is a finite dimensional $\mathbb{C}$ vector space. Now Grothendieck local duality (cf. \cite{HartshorneRD} V.Corollary 6.5) gives $\text{Ext}^4_\mathcal{O}(R^1\pi_*\mathcal{F},\mathcal{O})\cong (R^2\pi_*\mathcal{F})'$, so we can identify $(R^2\pi_*\mathcal{F})'$ with $\text{Ext}^4_\mathcal{O}(R^1\pi_*\mathcal{F},\mathcal{O})$.

Since the relative dimension of $\pi$ is 2, we see that the spectral sequence has only three rows $j=0,-1,-2$ and five columns $i=0,1,2,3,4$. Then one can easily see that the spectral sequence stabilizes after the fourth step. 

Now there is no nontrivial differential pointing from $E_r^{4,-2}$; we see that the stable object $E_4^{4,-2}$ is a quotient of $\text{Ext}^4_\mathcal{O}(R^2\pi_*\mathcal{F},\mathcal{O})$. Now the convergence gives you an injection $E_4^{4,-3}\hookrightarrow\text{Ext}^2_\mathcal{\widetilde{O}}(\mathcal{F},\mathcal{\widetilde{O}})$. Composing with the quotient $ \text{Ext}^4_\mathcal{O}(R^2\pi_*\mathcal{F},\mathcal{O})\twoheadrightarrow E_4^{4,-2}$, we get the required natural map.
\end{proof}
Now we can give the full characterization of maximal Cohen-Macaulay sheaves on $\mathcal{N}_{3,1}$ using reflexive sheaves on $T^*\mathbb{P}^2$.
\begin{proposition}\label{prop4.2}
Suppose $\mathcal{F}$ is a reflexive sheaf on $T^*\mathbb{P}^2$. Then $\pi_*\mathcal{F}$ is reflexive. It is maximal Cohen-Macaulay if and only if
\begin{enumerate}
\item $R^1\pi_*\mathcal{F}=0$.
\item $\text{Ext}^1_\mathcal{\widetilde{O}}(\mathcal{F},\mathcal{\widetilde{O}})=0$.
\item The map $ (R^2\pi_*\mathcal{F})'=\text{Ext}^4_\mathcal{O}(R^2\pi_*\mathcal{F},\mathcal{O})\rightarrow\text{Ext}^2_\mathcal{\widetilde{O}}(\mathcal{F},\mathcal{\widetilde{O}})$ is an isomorphism.
\end{enumerate}
\end{proposition}
\begin{proof}
Now suppose $M=\pi_*\mathcal{F}$. Consider the spectral sequence in Lemma \ref{lem3.10}. It has five columns $i=0,1,2,3,4$ and three rows $j=0,-1,-2$. Again by projectivity of $\pi$, since $R^1\pi_*\mathcal{F}$ and $R^1\pi_*\mathcal{F}$ are both supported only on 0, they are finite-dimensional $\mathbb{C}$-vector spaces. Then, since $\mathcal{O}$ is Cohen-Macaulay, Applying \cite{Weibel1994} Theorem 4.4.8 to a simple filtration of $R^1\pi_*\mathcal{F}$ implies $\text{Ext}^i_\mathcal{O}(R^1\pi_*\mathcal{F},\mathcal{O})=\text{Ext}^i_\mathcal{O}(R^2\pi_*\mathcal{F},\mathcal{O})=0$ for $i=0,1,2,3$

With the cohomology vanishing discussed above, the spectral sequence in Lemma \ref{lem3.10} looks like
$$
\begin{tikzcd}[column sep=1.2em, row sep=0.9em]
(\pi_*\mathcal{F})^\vee 
& \text{Ext}^1_\mathcal{O}(\pi_*\mathcal{F},\mathcal{O})\arrow[rrrdd,"d_3"]
& \text{Ext}^2_\mathcal{O}(\pi_*\mathcal{F},\mathcal{O})\arrow[rrd,"d_2"]
& \text{Ext}^3_\mathcal{O}(\pi_*\mathcal{F},\mathcal{O})
& \text{Ext}^4_\mathcal{O}(\pi_*\mathcal{F},\mathcal{O}) \\
0 & 0 & 0 & 0 & \text{Ext}^4_\mathcal{O}(R^1\pi_*\mathcal{F},\mathcal{O}) \\
0 & 0 & 0 & 0 & \text{Ext}^4_\mathcal{O}(R^2\pi_*\mathcal{F},\mathcal{O})
\end{tikzcd}
$$
The convergence gives you $\text{Ext}^4_\mathcal{\widetilde{O}}(\mathcal{F},\mathcal{\widetilde{O}})=\text{Ext}^4_\mathcal{O}(\pi_*\mathcal{F},\mathcal{O})$
, $\text{Ext}^1_\mathcal{\widetilde{O}}(\mathcal{F},\mathcal{\widetilde{O}})=\ker d_3$, $(\pi_*\mathcal{F})^\vee=\text{Hom}(\mathcal{F},\mathcal{\widetilde{O}})=\pi_*\mathcal{F}^\vee$, and two short exact sequences
\begin{align}\notag
0\rightarrow\ker d_2\rightarrow\text{Ext}^2_\mathcal{\widetilde{O}}(\mathcal{F},\mathcal{\widetilde{O}})&\rightarrow\text{Ext}^4_\mathcal{O}(R^2\pi_*\mathcal{F},\mathcal{O})/\text{im}\,d_3\rightarrow 0\ ,\\\notag
0\rightarrow\text{Ext}^3_\mathcal{O}(\pi_*\mathcal{F},\mathcal{O})\rightarrow\text{Ext}^3_\mathcal{\widetilde{O}}(&\mathcal{F},\mathcal{\widetilde{O}})\rightarrow\text{Ext}^4_\mathcal{O}(R^1\pi_*\mathcal{F},\mathcal{O})/\text{im}\,d_2\rightarrow 0
\end{align}

Since we are working on any reflexive sheaf $\mathcal{F}$, those relations also hold for $\mathcal{F}^\vee$. In particular, $(\pi_*\mathcal{F}^\vee)^\vee=\text{Hom}(\mathcal{F}^\vee,\mathcal{\widetilde{O}})=\pi_*\mathcal{F}^{\vee\vee}=\pi_*\mathcal{F}$, so $\pi_*\mathcal{F}$ is reflexive.
Now $\mathcal{F}$ is reflexive. This proves the first statement.

Since $\mathcal{F}$ is reflexive, by Lemma \ref{lem3.11}, $\text{Ext}^3_\mathcal{\widetilde{O}}(\mathcal{F},\mathcal{\widetilde{O}})=\text{Ext}^4_\mathcal{\widetilde{O}}(\mathcal{F},\mathcal{\widetilde{O}})=0$. Then, by the relations above, it gives you $\text{Ext}^3_\mathcal{O}(\pi_*\mathcal{F},\mathcal{O})=\text{Ext}^4_\mathcal{O}(\pi_*\mathcal{F},\mathcal{O})=0$ and $d_2$ surjective. Hence, our spectral sequence reduces again

$$
\begin{tikzcd}[column sep=1.2em, row sep=0.9em]
(\pi_*\mathcal{F})^\vee 
& \text{Ext}^1_\mathcal{O}(\pi_*\mathcal{F},\mathcal{O})\arrow[rrrdd,"d_3"]
& \text{Ext}^2_\mathcal{O}(\pi_*\mathcal{F},\mathcal{O})\arrow[rrd,"d_2",two heads]
& 0
& 0 \\
0 & 0 & 0 & 0 & \text{Ext}^4_\mathcal{O}(R^1\pi_*\mathcal{F},\mathcal{O}) \\
0 & 0 & 0 & 0 & \text{Ext}^4_\mathcal{O}(R^2\pi_*\mathcal{F},\mathcal{O})
\end{tikzcd}
$$

By Lemma \ref{lem3.3}, in this case, since we already have $\text{Ext}^3_\mathcal{\widetilde{O}}(\mathcal{F},\mathcal{\widetilde{O}})=\text{Ext}^4_\mathcal{\widetilde{O}}(\mathcal{F},\mathcal{\widetilde{O}})=0$, we see that $\pi_*\mathcal{F}$ is maximal Cohen-Macaulay if and only if $\text{Ext}^1_\mathcal{O}(\pi_*\mathcal{F},\mathcal{O})=\text{Ext}^2_\mathcal{O}(\pi_*\mathcal{F},\mathcal{O})=0$. Now we prove the equivalence between this and the conditions in this proposition.

Suppose first $\text{Ext}^1_\mathcal{O}(\pi_*\mathcal{F},\mathcal{O})=\text{Ext}^2_\mathcal{O}(\pi_*\mathcal{F},\mathcal{O})=0$. Since $d_2$ is surjective, we see that $\text{Ext}^4_\mathcal{O}(R^1\pi_*\mathcal{F},\mathcal{O})=0$. Since $R^1\pi_*\mathcal{F}$ is supported only at the closed point $0$, by Grothendieck local duality, we have $R^1\pi_*\mathcal{F}=0$, which gives (1). $\text{Ext}^1_\mathcal{\widetilde{O}}(\mathcal{F},\mathcal{\widetilde{O}})=\ker d_3\subseteq 0$, so $\text{Ext}^1_\mathcal{\widetilde{O}}(\mathcal{F},\mathcal{\widetilde{O}})=0$, which gives (2). Finally, (3) follows by the first short exact sequence above and that $\ker d_2=\text{im}\,d_3=0$, since they are all subquotients of $\text{Ext}^1_\mathcal{O}(\pi_*\mathcal{F},\mathcal{O})\ \text{or}\ \text{Ext}^2_\mathcal{O}(\pi_*\mathcal{F},\mathcal{O})$.

Now suppose conversely (1), (2), (3) hold.  By (3) and the first short exact sequence, we see that $\ker d_2=\text{im}\,d_3=0$. (1) implies $\text{Ext}^4_\mathcal{O}(R^1\pi_*\mathcal{F},\mathcal{O})=0$, so $\text{im}\,d_2=0$. (2) implies $\ker d_3=\text{Ext}^1_\mathcal{\widetilde{O}}(\mathcal{F},\mathcal{\widetilde{O}})$. Then one sees that $\text{Ext}^1_\mathcal{O}(\pi_*\mathcal{F},\mathcal{O})=\text{Ext}^2_\mathcal{O}(\pi_*\mathcal{F},\mathcal{O})=0$, completing the proof.
\end{proof}

This proposition allows us to turn the classification or construction problems of maximal Cohen-Macaulay sheaves on the singular variety $\mathcal{N}_{3,1}$ into the respective problems of reflexive sheaves on the smooth variety $T^*\mathbb{P}^2$ that satisfy conditions (1),(2),(3)
. It plays a similar role to Lemma 1.1 in \cite{ArtinVerdierEsnault1985}, while it is more complicated as we are no longer working on surface but on higher dimensional symplectic singularities.

For vector bundles on $T^*\mathbb{P}^2$, the conditions in Proposition \ref{prop4.2} are cleaner.
\begin{corollary}
For a vector bundle $\mathcal{F}$ on $T^*\mathbb{P}^2$, $\pi_*\mathcal{F}$ is maximal Cohen-Macaulay if and only if 
\begin{enumerate}
\item $R^1\pi_*\mathcal{F}=R^1\pi_*\mathcal{F}^\vee=0$
\item The map induced by the spectral sequence $ (R^2\pi_*\mathcal{F})'\rightarrow R^2\pi_*\mathcal{F}^\vee$ is an isomorphism.
\end{enumerate}
\end{corollary}

\subsection{Bundles on $\mathbb{P}^2$ pushing forward to maximal Cohen-Macaulay sheaves}
For construction purpose, we are more interested in the conditions described in Proposition \ref{prop3.14}, that is, vector bundles with $R^{>0}\pi_*\mathcal{F}=R^{>0}\pi_*\mathcal{F}^\vee=0$. To construct such vector bundles on $T*\mathbb{P}^2$, it's natural to consider the vector bundles pulled back from $\mathbb{P}^2$. The following proposition characterizes those sheaves satisfying our condition. 
\begin{proposition}\label{prop4.4}
Suppose $f:T^*\mathbb{P}^2\rightarrow\mathbb{P}^2$ is the projection map and $\mathcal{O}_{\mathbb{P}^2}(1)$ is the sheaf of linear forms on $\mathbb{P}^2$. Then, for any vector bundle $E$ on $\mathbb{P}^2$, we have
$$R^{>0}\pi_*f^*E=0 \Longleftrightarrow H^{>0}(E(n))=0\ \ \text{for} \ \ n\geq 0
$$
\end{proposition}
\begin{proof}
Since the target of $\pi$ is affine, $R^{>0}\pi_*f^*E=0$ is the same as  $H^{>0}(f^*E)=0$. Now since$f$ is affine, we have
$$H^{>0}(f^*E)=H^{>0}(f_*f^*E)=H^{>0}(S^\bullet (T_X)\otimes E)
$$
so the left hand side is equivalent to $H^{>0}(S^n(T_X)\otimes E)$ for $n\geq 0$. Now we prove this is equivalent to the right hand side.

Consider the Euler sequence
$$0\rightarrow\mathcal{O}_{\mathbb{P}^2}\rightarrow\mathcal{O}_{\mathbb{P}^2}(1)^{\oplus 2}\rightarrow T_{\mathbb{P}^2}\rightarrow 0
$$
Taking the symmetric product $S^n$ and tensoring $E$, one has the following filtration
$$E(n)^{\binom{n+2}{2}}=E\otimes S^n(\mathcal{O}_{\mathbb{P}^2}(1)^{\oplus 2})=F^0\supseteq F^1\supseteq...\supseteq F^n\supseteq F^{n+1}=0
$$
with quotients 
$$F^p/F^{p+1}\cong E\otimes S^p(T_{\mathbb{P}^2})\otimes S^{n-p}(\mathcal{O}_{\mathbb{P}^2})=E\otimes S^p(T_{\mathbb{P}^2})
$$

If $H^{>0}(S^n(T_X)\otimes E)=0$ for all $n\geq0$, then all quotients in this filtration have vanishing $H^{>0}$, so we also have $H^{>0}(E(n)^{\binom{n+2}{2}})=0$; hence $H^{>0}(E(n))=0$ for any $n\geq 0$.

Conversely, if $H^{>0}(E(n))=0$ for any $n\geq 0$, we prove by induction that $H^{>0}(S^n(T_X)\otimes E)=0$ for all $n\geq0$. For $n=0$, this is trivial. Now suppose $H^{>0}(S^k(T_X)\otimes E)=0$ for all $k\leq n-1$. For $k=n$, we see in the filtration above that there is a short exact sequence
$$0\rightarrow F_1\rightarrow E(n)^{\binom{n+2}{2}}\rightarrow E\otimes S^n(T_{\mathbb{P}^2})\rightarrow 0
$$
Now $H^{>0}(E(n)^{\binom{n+2}{2}})=0$. By induction and filtration of $F_1$, we have $H^{>0}(F_1)=0$ as well; hence $H^{>0}(E\otimes S^n(T_{\mathbb{P}^2}))=0$, completing the induction step.
\end{proof}
\begin{corollary}\label{coro4.5}
Suppose $E$ is a vector bundle on $\mathbb{P}^2$, then 
$$R^{>0}\pi_*f^*E=R^{>0}\pi_*f^*E^\vee=0 \Longleftrightarrow H^{>0}(E(n))=H^{<2}(E(-n-3))=0\ \,\text{for}\ n\geq 0
$$
and in this case $\pi_*f^*E$ is maximal Cohen-Macaulay.
\end{corollary}
\begin{proof}
This is immediately implied by Proposition \ref{prop4.4}, Proposition \ref{prop3.14}, and Serre duality.
\end{proof}
The following proposition tells us the relationship between the indecomposability of $E$ and the indecomposability of $\pi_*f^*E$.
\begin{proposition}\label{prop4.6}
Suppose $E$ is a vector bundle on $\mathbb{P}^2$. Then $E$ is indecomposable if and only if $\pi_*f^*E$ cannot be decomposed into the sum of two modules, both of rank $\geq 1$.
\end{proposition}
\begin{proof}
The "if" part is obvious. Now suppose $\pi_*f^*E=M\oplus N$; then, by pulling back and taking the double dual, one has $\overline{\pi_*f^*E}=\overline{M}\oplus \overline{N}$. By Proposition \ref{prop3.13} (2), we see that $\overline{\pi_*f^*E}=f^*E$; so $f^*E=\overline{M}\oplus \overline{N}$. Now, since $\overline{M}$ and $\overline{N}$ are reflexive, they are torsion free. Suppose $\eta$ is the generic point of the exceptional fiber $\mathbb{P}^2$; then $\overline{M}_\eta\neq 0$, $\overline{N}_\eta\neq 0$. Therefore, by Nakayama's lemma, $\overline{M}|_{\mathbb{P}^2}$ and $\overline{N}|_{\mathbb{P}^2}$ still have positive ranks.

Now $f^*E|_{\mathbb{P}^2}=\overline{M}|_{\mathbb{P}^2}\oplus \overline{N}|_{\mathbb{P}^2}$. Regarding the exceptional fiber as the zero section of $\mathbb{P}^2$ in its cotangent bundle, we see that $f^*E|_{\mathbb{P}^2}=E$, so $E=\overline{M}|_{\mathbb{P}^2}\oplus \overline{N}|_{\mathbb{P}^2}$. Notice that the summand of a vector bundle is again a vector bundle because, over a local ring, the summand of a free module is again free. Since $\overline{M}|_{\mathbb{P}^2}$ and $\overline{N}|_{\mathbb{P}^2}$ are nontrivial, we see that $E$ is decomposable, completing the proof of another direction.
\end{proof}

This proposition allows us to construct indecomposable maximal Cohen-Macaulay modules using indecomposable bundles on $\mathbb{P}^2$. From now on, for any integer $r>0$, we are interested in the vector bundle $E$ on $\mathbb{P}^2$ satisfying the following conditions
\begin{equation}\tag{$\ast$}\label{star}
E \ \text{is indecomposable}, \  H^{>0}(E(n))=H^{<2}(E(-n-3))=0\ \text{for} \ n\geq 0
\end{equation}
As shown in Corollary \ref{coro4.5} and Proposition \ref{prop4.6}, those vector bundles can provide us with various examples of indecomposable maximal Cohen-Macaulay sheaves on $\mathcal{N}_{3,1}$, so we aim to construct such bundles.

We start from rank 1 bundles: for $E=\mathcal{O}_{\mathbb{P}^2}(a)$, we can easily see that $E$ satisfies $(\ast)$ if and only if $-2\leq a\leq 2$. This implies $E$ is a somewhat "balanced" bundle. The following lemma explains what it means for $E$ to be "balanced" when the rank is arbitrary.
\begin{lemma}
Suppose $E$ is a vector bundle on $\mathbb{P}^2$ satisfying $H^{>0}(E(n))=H^{<2}(E(-n-3))=0$ for $n\geq 0$. For any line $L\cong\mathbb{P}^1$ on $\mathbb{P}^2$, the restriction of $E$ to $L$ splits as
$$E=\bigoplus_{i=1}^{r}\mathcal{O}_L(e_i)
$$
and we have $-2\leq e_i\leq 2$.
\end{lemma}
\begin{proof}
The splitting statement is obvious since all vector bundles on $\mathbb{P}^1$ split into the sum of line bundles. It suffices for us to prove $-2\leq e_i\leq 2$. Consider the exact sequence
$$0\rightarrow\mathcal{O}_{\mathbb{P}^2}(-1)\rightarrow\mathcal{O}_{\mathbb{P}^2}\rightarrow\mathcal{O}_L\rightarrow0
$$
Tensoring with $E(n)$ gives you $0\rightarrow E(n-1)\rightarrow E(n)\rightarrow E(n)|_L\rightarrow 0$. Now, take the long exact sequence of cohomology. By cohomology vanishing under the assumption, one sees that $H^1(E(n)|_L)=0$ for $n\geq 1$ and $H^0(E(n)|_L)=0$ for $n\leq -2$. If we suppose $E$ splits as $\bigoplus_{i=1}^{r}\mathcal{O}_L(e_i)$, then this implies $H^1(\mathcal{O}_L(e_i+1))=0$ and $H^0(\mathcal{O}_L(e_i-3))=0$. Then one sees easily that we must have $-2\leq e_i\leq 2$.
\end{proof}

According to this lemma, any vector bundle with $H^{>0}(E(n))=H^{<2}(E(-n-3))=0$ for $n\geq 0$ splits as $\bigoplus_{i=-2}^{2}\mathcal{O}_L(i)^{\oplus a_i}$ when restricted to any line $L\cong\mathbb{P}^1$. One can easily show the following relations:
\begin{enumerate}
\item $a_2+a_1+a_0+a_{-1}+a_{-2}= r$, where $r$ is the rank of $E$.
\item $2a_2+a_1-a_{-1}-2a_{-2}= c_1(E)$, the first Chern class of $E$.
\item $a_2=H^0(E(-2)), \ a_{-2}=H^2(E(-1))$.
\end{enumerate}

In the discussion above, we only care about the cohomology vanishing conditions. The following lemma tells us the restriction that indecomposability gives us.
\begin{lemma}\label{lem4.8}
Let $E$ and $F$ be vector bundles on $\mathbb{P}^2$ and let $L \subset \mathbb{P}^2$ be a line. Suppose
$$
E|_L \cong F|_L
\quad \text{and} \quad
\mathrm{Ext}^1_{\mathbb{P}^2}(F,E(-1))=0.
$$
Then $E \cong F$.
\end{lemma}
\begin{proof}
We want to extend the isomorphism $\varphi: F|_L \to E|_L$ to one over $\mathbb{P}^2$. This is possible since there is an exact sequence
$$
0 \to \text{Hom}_{\mathbb{P}^2}(F,E(-1))
\to \text{Hom}_{\mathbb{P}^2}(F,E)
\to \text{Hom}_L(F|_L,E|_L)
\to \text{Ext}^1_{\mathbb{P}^2}(F,E(-1))=0.
$$
Thus every morphism $F|_L \to E|_L$ extends to a morphism $f:F\to E$ satisfying $f|_L=\varphi$.

Consider the induced morphism on determinants $\det f : \det F \to \det E$. Since $E|_L \cong F|_L$, the bundles $E$ and $F$ have the same rank and the same first Chern class. Hence $\det f$ is multiplication by a constant in $\mathbb{C}$. This constant is nonzero because $\det f$ does not vanish on $L$. It follows that $f$ is an isomorphism everywhere on $\mathbb{P}^2$. Hence $E \cong F$.
\end{proof}
\begin{corollary}\label{coro4.9}
Suppose $E$ is a non-splitting vector bundle on $\mathbb{P}^2$ with $H^{>0}(E(n))=H^{<2}(E(-n-3))=0$ for $n\geq 0$. Suppose its restriction to a line $L$ splits as $\bigoplus_{i=-2}^{2}\mathcal{O}_L(i)^{\oplus a_i}$. Then we have
$$a_{-1}+a_0>0,\ a_0+a_1>0
$$
\end{corollary}
\begin{proof}
Since $E$ is non-splitting, $E\neq \bigoplus_{i=-2}^{2}\mathcal{O}_{\mathbb{P}^2}(i)^{\oplus a_i}$; thus, by Lemma \ref{lem4.8}, we should have
$$\mathrm{Ext}^1_{\mathbb{P}^2}(\bigoplus_{i=-2}^{2}\mathcal{O}_{\mathbb{P}^2}(i)^{\oplus a_i},E(-1))\neq0,\ \ \mathrm{Ext}^1_{\mathbb{P}^2}(E,\bigoplus_{i=-2}^{2}\mathcal{O}_{\mathbb{P}^2}(i-1)^{\oplus a_i})\neq0
$$
Now $\mathrm{Ext}^1_{\mathbb{P}^2}(\bigoplus_{i=-2}^{2}\mathcal{O}_{\mathbb{P}^2}(i)^{\oplus a_i},E(-1))=\bigoplus_{i=-2}^{2}H^1(E(-1-i))^{\oplus a_i}$ is nonvanishing. By the cohomology vanishing assumption, $H^1(E(-1-i))=0$ if $i=2,-1,-2$, hence $a_1$ and $a_0$ cannot be all zero. This proves $a_0+a_1>0$. Similarly, $\mathrm{Ext}^1_{\mathbb{P}^2}(E,\bigoplus_{i=-2}^{2}\mathcal{O}_{\mathbb{P}^2}(i-1)^{\oplus a_i})=\bigoplus_{i=-2}^{2}H^1(E(-i-2))^{\oplus a_i}$ by Serre duality, and this is nonvanishing. Now $H^1(E(-2-i))=0$ if $i=-2,1,2$, so we should have $a_{-1}+a_0>0$, completing the proof.
\end{proof}

\subsection{Classification of 2-bundles with cohomology vanishing}
Using Corollary \ref{coro4.9}, we can try to classify all rank 2 indecomposable vector bundles on $\mathbb{P}^2$ satisfying our cohomology vanishing condition.

Now we suppose $E$ is a rank 2 indecomposable bundle with $H^{>0}(E(n))=H^{<2}(E(-n-3))=0$ for $n\geq 0$. By Corollary \ref{coro4.9},the splitting type of $E|_L$ for any line $L$ can only be one of the following: $(a,b)=(0,2), (0,1), (0,0), (0,-1), (0,-2), (1,-1)$ (here splitting type (a,b) means $E|_L=\mathcal{O}_L(a)\oplus\mathcal{O}_L(b)$). Hence, $c_1(E)=a+b\in\{-2,-1,0,1,2\}$. Now we classify those bundles by their Chern classes.

First, if $c_1(E)\neq 0$, we see that the restriction of $E$ to any line is $\mathcal{O}_L\oplus\mathcal{O}_L(c_1(E))$. Thus, $E$ is a \emph{uniform bundle}, i.e., its restriction to any projective line has the same splitting type. By a theorem of Van de Ven (cf. \cite{OSS1980} Theorem 2.2.2),  a uniform 2-bundle on $\mathbb{P}^2$ that does not split has the form $T_{\mathbb{P}^2}(m)$. Since $c_1(T_{\mathbb{P}^2}(m))=3-2m$, we see that $m=-1$ or $m=-2$ are the only possibilities. So $E=T_{\mathbb{P}^2}(-1)$ or $T_{\mathbb{P}^2}(-2)$. Checking the cohomology table of $T_{\mathbb{P}^2}$,
$$
\begin{array}{c|cccccccccccc}
k & -8 & -7 & -6 & -5 & -4 & -3 & -2 & -1 & 0 & 1 & 2 & 3\\ \hline
h^2(T(k)) & 24 & 15 & 8 & 3 & 0 & 0 & 0 & 0 & 0 & 0 & 0 & 0\\
h^1(T(k)) & 0  & 0  & 0 & 0 & 0 & 1 & 0 & 0 & 0 & 0 & 0 & 0\\
h^0(T(k)) & 0  & 0  & 0 & 0 & 0 & 0 & 0 & 3 & 8 & 15 & 24 & 35
\end{array}
$$
we see that both $T_{\mathbb{P}^2}(-1)$ and $T_{\mathbb{P}^2}(-2)$ satisfy our cohomology vanishing condition. So we have proved the following.
\begin{proposition}
Suppose $E$ is a non-splitting 2-bundle on $\mathbb{P}^2$ with $H^{>0}(E(n))=H^{<2}(E(-n-3))=0$ for $n\geq 0$. If $c_1(E)\neq0$, then $E=T_{\mathbb{P}^2}(-1)$ or $T_{\mathbb{P}^2}(-2)$.
\end{proposition}

Now it suffices to classify the $c_1(E)=0$ case. By Riemann-Roch for surfaces, we have the following computation (cf. \cite{EisenbudHarris2016}, Theorem 14.2):
\begin{align}\notag
\chi(E(n))&=\frac{c_1(E(n))^2-2c_2(E(n))+c_1(E(n))c_1(T_{\mathbb{P}^2})}{2}+\text{rank} (E(n))\frac{c_1(T_{\mathbb{P}^2})^2+c_2(T_{\mathbb{P}^2})}{12}\\\notag
&=\frac{r}{2}n^2+(\frac{3}{2}r+c_1(E))n+\frac{1}{2}c_1(E)^2-c_2+\frac{3}{2}c_1(E)+r\\\notag
&=n^2+3n+2-c_2
\end{align}
Under our cohomology vanishing condition, $\chi(E)=H^0(E)>0$, so $2-c_2\geq 0$, i.e., $c_2\leq 2$. 

On the other hand, the splitting type of $E$ restricted to any line $L$ is $\mathcal{O}_L\oplus\mathcal{O}_L$ or $\mathcal{O}_L(1)\oplus\mathcal{O}_L(-1)$. Using \cite{OSS1980} Theorem 2.25 or Lemma 3.2.2, we see that the lines on $\mathbb{P}^2$ on which $E$ splits as $\mathcal{O}_L(1)\oplus\mathcal{O}_L(-1)$ form a curve in $(\mathbb{P}^2)^*$, which is called the curve of \emph{jumping lines}. By \cite{EisenbudHarris2016} Proposition 14.8, the degree of this curve is $c_2(E)-\frac{1}{2}c_1(E)^2=c_2$; hence $c_2\geq 0$. Therefore $c_2\in\{0,1,2\}$.

If $c_2=0$, we see that there are no jumping lines, so $E$ is a uniform bundle. Since all non-splitting uniform bundles are of the form $T_{\mathbb{P}^2}(m)$, and we see that none of them has $c_1=0$, so $E$ must split.

If $c_2=1$, we have $H^0(E)=\chi(E)=1$. We can suppose $s$ is a nonzero global section that generates $H^0(E)$. Now let's prove that $s$ vanishes in codimension 2.

For any line $L$ on $\mathbb{P}^2$, since $E|_L=\mathcal{O}_L\oplus\mathcal{O}_L\ \text{or}\ \mathcal{O}_L(1)\oplus\mathcal{O}_L(-1)$, we see that  $E(1)|_L$ is generated by global sections. Now applying the long exact sequence of cohomology to $0\rightarrow E\rightarrow E(1)\rightarrow E(1)|_L\rightarrow 0$,
$$\cdots\rightarrow H^0(E(1))\rightarrow H^0(E(1)|_L)\rightarrow H^1(E)=0\rightarrow\cdots
$$
We see that $H^0(E(1))\rightarrow H^0(E(1)|_L)$ is surjective, hence $E(1)$ is generated by global sections. By Lemma 5.2 in \cite{EisenbudHarris2016}, we can choose two general sections $u,v$ of $E(1)$ such that they both vanish in codimension 2 and the degeneracy locus of $u\wedge v$ is a generically reduced curve $C$. Indeed, by Bertini's Theorem, we can choose $C$ to be nonsingular, hence a smooth quadratic curve since $c_1(E(1))=2$. Then $u,v$ induces a short exact sequence
$$0 \rightarrow\mathcal{O}_{\mathbb{P}^2}^{\oplus 2}\xrightarrow{u,v} E(1)\rightarrow Q\rightarrow 0
$$
Now take the long exact sequence of local cohomology at any closed point $m$,
$$
\cdots\rightarrow H^0_m(E(1))=0\rightarrow H^{0}_m(Q)\rightarrow H^1_m(\mathcal{O}_{\mathbb{P}^2})^{\oplus 2}=0
$$
We see that none of the sections of $Q$ are supported on closed points. Then $Q|_C$ is torsion free, hence locally free on curve $C$. Now $u,v$ vanishing in codimension 2 tells us that the rank of $Q|_C$ is 1. Since smooth quadratic curves on $\mathbb{P}^2$ are all rational, we see that  $Q=\mathcal{O}_C(d)$ for some d (here $\mathcal{O}_{\mathbb{P}^2}(1)|_L=\mathcal{O}_C(2)$). Computing the dimension of global sections, we see that $h^0(\mathcal{O}_C(d))=h^0(Q)=h^0(E(1))-h^0(\mathcal{O}_{\mathbb{P}^2}^{\oplus 2})=1^2+3+1-2=3$, so $d=2$. Twisting by $\mathcal{O}_{\mathbb{P}^2}(-1)$ gives us a short exact sequence
$$0\rightarrow \mathcal{O}_{\mathbb{P}^2}(-1)^{\oplus 2}\rightarrow E\rightarrow \mathcal{O}_C\rightarrow 0
$$
We see that the global section $s$ is mapping to the constant $\lambda\neq0$ in $\mathcal{O}_C$. If $s$ degenerates in codimension 1, then the degeneracy locus must intersect $C$, a contradiction. Hence $s$ vanishes only at closed points. 

Now $s$ induces an exact sequence
$$0\rightarrow\mathcal{O}_{\mathbb{P}^2}\xrightarrow{s} E\rightarrow G\rightarrow 0
$$
Again by taking the long exact sequence of local cohomology, we see that $G$ has no section supports at closed points. For codimension 1 points, we see that $s$ is nonvanishing, so $\mathcal{O}_{\mathbb{P}^2}$ splits into $E$ at codimension 1 points. Then $G$ is locally free in codimension 1; hence, it has no sections supported at codimension 1 points. We see $G$ is torsion free since it has no sections supported on proper subvarieties. The injection $G\hookrightarrow G^{\vee\vee}=\mathcal{O}_{\mathbb{P}^2}(c_1(G))=\mathcal{O}_{\mathbb{P}^2}$ implies $G$ is an ideal sheaf. By computing Chern classes
$$c(G)=\frac{c(E)}{c(\mathcal{O}_{\mathbb{P}^2})}=c(E),\quad\text{so}\ c_2(G)=1
$$
we see that $G=I_x$ is the ideal sheaf of one point, so E is the extension of a one point ideal sheaf $I_x$ by $\mathcal{O}_{\mathbb{P}^2}$.

Now we check all such extensions. Since you have a short exact sequence
$$0\rightarrow I_x\rightarrow\mathcal{O}_{\mathbb{P}^2}\rightarrow \mathcal{O}_x\rightarrow 0
$$
Applying $\text{Hom}_{\mathbb{P}^2}(-,\mathcal{O}_{\mathbb{P}^2})$, we see that 
$$\text{Ext}^1_{\mathbb{P}^2}(I_x,\mathcal{O}_{\mathbb{P}^2})=\text{Ext}^2_{\mathbb{P}^2}(\mathcal{O}_x,\mathcal{O}_{\mathbb{P}^2})\cong H^0(\mathcal{O}_x)\cong \mathbb{C}
$$
Then, for any $x\in\mathbb{P}^2$, there is a one dimensional family of extensions, hence a unique isomorphism class of extension sheaf $E_x$ of $I_x$ by $\mathcal{O}_{\mathbb{P}^2}$.  Let's check if $E_x$ is a vector bundle and compute its cohomology table.

To check whether $E_x$ is locally free, it suffices to check whether $\mathcal{E}xt^1(E_x,\mathcal{O})=\mathcal{E}xt^2(E_x,\mathcal{O})=0$. Using the exact sequence defining $E_x$ and the exact sequence above, we see that
\begin{align}\notag
\mathcal{E}xt^1_{\mathbb{P}^2}(E_x,\mathcal{O}_{\mathbb{P}^2})&=\text{coker}\ (\mathcal{H}om_{\mathbb{P}^2}(\mathcal{O}_{\mathbb{P}^2},\mathcal{O}_{\mathbb{P}^2})\rightarrow\mathcal{E}xt^1(I_x,\mathcal{O}_{\mathbb{P}^2}))=0\\\notag
\mathcal{E}xt^2_{\mathbb{P}^2}(E_x,\mathcal{O}_{\mathbb{P}^2})&=\mathcal{E}xt^2_{\mathbb{P}^2}(I_x,\mathcal{O}_{\mathbb{P}^2})=0
\end{align}
The first equals zero because the extension class is nontrivial, so $\mathcal{H}om_{\mathbb{P}^2}(\mathcal{O}_{\mathbb{P}^2},\mathcal{O}_{\mathbb{P}^2})=\mathcal{O}_{\mathbb{P}^2}$ is mapping to the generator of $\mathcal{E}xt^1(I_x,\mathcal{O}_{\mathbb{P}^2})=\mathcal{E}xt^2(\mathcal{O}_x,\mathcal{O}_{\mathbb{P}^2})=\mathcal{O}_x$. Hence, $E_x$ is a vector bundle. 

Finally, let's check the cohomology vanishing condition using the exact sequence 
$$0\rightarrow\mathcal{O}_{\mathbb{P}^2}(n)\rightarrow E_x(n)\rightarrow I_x(n)\rightarrow 0
$$
Let's check $H^{>0}(E_x(n))=0$ for $n\geq 0$ first. For $n\geq 0$, by applying the long exact sequence of cohomology to the exact sequence above, we have
\begin{align}\notag
H^1(E_x(n))&=H^1(I_x(n))=\text{coker}\,H^0(\mathcal{O}_{\mathbb{P}^2}(n))\rightarrow H^0(\mathcal{O}_x(n))=0\\\notag
H^2(E_x(n))&=H^2(I_x(n))=H^1(\mathcal{O}_x(n))=0
\end{align}
so $H^{>0}(E_x(n))=0$ for $n\geq 0$. Now $H^{i}(E_x(-n-3))=H^{2-i}(E_x^\vee(n))$. Since $c_1(E_x)=0$, we find that $E_x^\vee\cong E_x\otimes (\det E_x)^{-1}\cong E_x$, so $H^{<0}(E_x(-n-3))=0 \Longleftrightarrow H^{>0}(E_x(n))=0$. Then we see $H^{<0}(E_x(-n-3))=0$ for $n\geq 0$ as well. In summary, we have proved
\begin{proposition}\label{prop4.11}
Suppose $E$ is a non-splitting 2-bundle on $\mathbb{P}^2$ with $H^{>0}(E(n))=H^{<2}(E(-n-3))=0$ for $n\geq 0$. If $c_1(E)=0,\ c_2(E)=1$, then $E=E_x$, where $E_x$ is determined by the unique non-splitting bundle defined by 
$$0\rightarrow\mathcal{O}_{\mathbb{P}^2}\rightarrow E_x\rightarrow I_x\rightarrow 0
$$
Hence, there is a family of such bundles parametrized by points on $\mathbb{P}^2$.
\end{proposition}

Finally, let's classify the case $c_1(E)=0$ and $c_2(E)=2$. Let's first introduce stability: We call a vector bundle $E$ on $\mathbb{P}^2$ \emph{stable} if, for any coherent subsheaf $0\neq F\subset E$ with $0<\text{rank}F<\text{rank} E$, we have 
$$
\mu(F)=\frac{c_1(F)}{\text{rank}F}<\mu(E)=\frac{c_1(E)}{\text{rank}E}
$$
Given a vector bundle $E$, define $E_{\text{norm}}$ to be the unique twist $E(k_E)$ such that $-\text{rank}(E)+1\leq E\leq 0$. One has the following criterion (cf. \cite{Hoppe1984}, Lemma 2.6):
\begin{proposition}[Hoppe criterion]
Let $E$ be a rank $r$ vector bundle over a cyclic projective variety. 
If
$$
H^0\big((\wedge^q E)_{\mathrm{norm}}\big)=0
\quad \text{for } 1\le q\le r-1,
$$
then $E$ is stable. 
\end{proposition}
In our case, $\mathbb{P}^2$ is cyclic projective; $c_1(E)=0$, so $E_{\text{norm}}$. By computation of the Euler characteristic above: $H^0(E)=\chi(E)=0$, so by the Hoppe criterion, $E$ is a stable 2-bundle.

Now we show that any stable 2-bundle $E$ with $c_1(E)=0$ and $c_2(E)=2$ satisfies our cohomology vanishing condition. For such $E$, firstly, we have $H^0(E)=0$ since any global section will induce $\mathcal{O}_{\mathbb{P}^2}$ a subsheaf of $E$. However, $\mu(\mathcal{O}_{\mathbb{P}^2})=0=\mu(E)$ will make $E$ unstable. Now, for any $n\geq 0$, there is an injection $\mathcal{O}_{\mathbb{P}^2}\hookrightarrow\mathcal{O}_{\mathbb{P}^2}(n)$ by choosing any global section of $\mathcal{O}_{\mathbb{P}^2}(n)$. Twisted by $E(-n)$, the injection becomes $E(-n)\hookrightarrow E$. Since $H^0(E)=0$, we see that $H^0(E(-n))=0$ as well for $n\geq 0$.

On the other hand, $E^\vee=E\otimes(\det E)^{-1}=E$ since $c_1(E)=0$. By Serre duality, $H^2(E(-1))=H^0(E^\vee(-2))=H^0(E(-2))=0$, $H^2(E)=H^0(E^\vee(-3))=H^0(E(-3))=0$. By the computation of the Euler characteristic above, we see that $\chi(E(n))=n^2+3n$, so $\chi(E)=0$. Now $H^0(E)=H^2(E)=0$, we see that $H^1(E)=0$ as well.

Now we have $H^2(E(-1))=H^1(E)=0$, and this gives Castelnuovo–Mumford 1-regularity of $E$ (cf. \cite{EisenbudSyzygies} Theorem 4.2). so for $n\geq 0$ , there is $H^1(E(n))=H^2(E(n))=0$. Now since $E^\vee=E$, Serre duality also gives $H^{<0}(E_x(-n-3))=0$ for $n\geq 0$. Hence, we have proved

\begin{proposition}\label{prop4.13}
The family of non-splitting 2-bundles $E$ on $\mathbb{P}^2$ with $H^{>0}(E(n))=H^{<2}(E(-n-3))=0$ for $n\geq 0$ and $c_1(E)=0,\ c_2(E)=2$ coincides with the family of all stable 2-bundles with $c_1(E)=0,\ c_2(E)=2$.
\end{proposition}

In \cite{OSS1980} Example 4.3.2, the moduli space $M_{\mathbb{P}^2}(0,2)$ of stable 2-bundles with $c_1(E)=0,\ c_2(E)=2$ is constructed. It is also proved here that there is a bijection $M_{\mathbb{P}^2}(0,2)\rightarrow \mathbb{P}(S_3^3)$ with the set of nonsingular conics in $(\mathbb{P}^2)^*$. This is a five dimensional variety.

In summary, we have shown the following theorem.
\begin{theorem}\label{theo4.14}
Suppose $E$ is a non-splitting 2-bundle on $\mathbb{P}^2$ with 
$$
H^{>0}(E(n))=H^{<2}(E(-n-3))=0\ \ \text{for}\ \ n\geq 0
$$
then one of the following holds:
\begin{enumerate}
\item $E=T_{\mathbb{P}^2}(-1)\ \text{or}\ T_{\mathbb{P}^2}(-2)$.
\item There is $x\in\mathbb{P}^2$ such that $E$ is the unique nontrivial extension of  $I_x$ by $\mathcal{O}_{\mathbb{P}^2}$.
\item $E$ is a stable 2-bundle with $c_1(E)=0$, $c_2(E)=2$.
\end{enumerate}
Conversely, any 2-bundle satisfying one of these three conditions has the cohomology vanishing above.
\end{theorem}

\subsection{Construction of higher rank maximal Cohen-Macaulay sheaves}Now we consider $E$ with rank $\geq 3$. It's natural to ask if we can construct some indecomposable vector bundle for each $r\in\mathbb{Z}^+$. For this purpose, we need to introduce a large family of bundles: the \emph{Steiner bundles}. A Steiner bundle on $\mathbb{P}^2$ is a bundle $E$ that admits a two step resolution
$$0\rightarrow\mathcal{O}_{\mathbb{P}^2}(-1)^{\oplus t}\rightarrow\mathcal{O}_{\mathbb{P}^2}^{\oplus t+r}\rightarrow E\rightarrow 0
$$
In order to construct higher rank indecomposable vector bundles satisfying our cohomology vanishing condition, we should ask three questions about the general quotient of the map (equivalently, a matrix of linear forms) $0\rightarrow\mathcal{O}_{\mathbb{P}^2}(-1)^{\oplus t}\rightarrow\mathcal{O}_{\mathbb{P}^2}^{\oplus t+r}$:
\begin{enumerate}
    \item When is a general quotient locally free?
    \item When is a general quotient indecomposable?
    \item When does a general quotient satisfy the cohomology vanishing condition above?
\end{enumerate}

We start from (1). By \cite{Harris1992} Proposition 12.2, we have the following proposition.
\begin{proposition}\label{prop4.15}
If $r\geq 2$, a general quotient of $\mathcal{O}_{\mathbb{P}^2}(-1)^{\oplus t}\rightarrow\mathcal{O}_{\mathbb{P}^2}^{\oplus t+r}$ is locally free.
\end{proposition}
For (2), the lack of a useful criterion of indecomposability inspires us to consider a stronger condition, namely stability. It's obvious that a stable bundle cannot be decomposable. Now \cite{coskun2025stability} Theorem 5.1 gives you the following proposition.
\begin{proposition}\label{prop4.16}
If $ 2\leq r\leq \frac{1+\sqrt{5}}{2}t$, a general quotient of $\mathcal{O}_{\mathbb{P}^2}(-1)^{\oplus t}\rightarrow\mathcal{O}_{\mathbb{P}^2}^{\oplus t+r}$ is stable.
\end{proposition}

Finally, as explained in \cite{coskun2025stability}, the defining resolution of a Steiner bundle
is the Gaeta resolution for $E$ . Hence, $E$ is a general member of the stack of prioritary sheaves with
Chern character $\mathcal{C}h(E)$ Since the general bundle in these stacks has at most one nonzero cohomology
group, we have the following proposition.
\begin{proposition}\label{prop4.17}
If $r\geq 2$, a general quotient of $\mathcal{O}_{\mathbb{P}^2}(-1)^{\oplus t}\rightarrow\mathcal{O}_{\mathbb{P}^2}^{\oplus t+r}$ has natural cohomology.
\end{proposition}
Now, in order to further control the cohomology of the Steiner bundle, we apply the long exact sequence of cohomology to its defining exact sequence. It's easy to see that $H^{>0}(E(n))=0$ for $n\geq -1$. However, by computing the Euler characteristic:
\begin{align}\notag
\chi(E(n))&=(r+t)\chi(\mathcal{O}_{\mathbb{P}^2}(n))-t\chi(\mathcal{O}_{\mathbb{P}^2}(n-1))\\\notag
&=(r+t)\frac{(n+2)(n+1)}{2}-t\frac{n(n+1)}{2}\\\notag
&=\frac{r}{2}n^2+(\frac{3}{2}+t)n+r+t
\end{align}
One sees that $\chi(E(-3))=r-2t$, and it is negative if $r\leq \frac{1+\sqrt{5}}{2}t$. Hence, we turn to consider if $E(-1)$ could satisfy our cohomology vanishing.

Let $n=-4$, we see that $\chi(E(-4))=3r-3t$. Now $\chi(E(-2))=-t<0$, and the function $\chi(E(n))$ is a quadratic function, we can ensure $\chi(E(n))>0$ for any $n\leq -4$ if $\chi(E(-4))=3r-3t>0$. Then the naturality of the cohomology tells us $H^{<2}(E(n))=0$ for $n\leq -4$.

To summarize our idea, replacing $E$ with $E(-1)$, we have the following theorem:
\begin{proposition}\label{prop4.18}
Let positive integers $r,t$ satisfy $r\geq 2$ and $t< r<\frac{\sqrt5+1}{2}t$. Let $\phi$ be a general $(t+r)\times t$ matrix of linear forms on $\mathbb{P}^2$ and let $E$ be the quotient sheaf defined by the sequence
$$0\rightarrow\mathcal{O}_{\mathbb{P}^2}(-2)^t\xrightarrow{\phi}\mathcal{O}_{\mathbb{P}^2}(-1)^{t+r}\rightarrow E\rightarrow 0
$$
then $E$ is a stable $($hence indecomposable$)$ bundle with $H^{>0}(E(n))=H^{<2}(E(-n-3))=0$ for $n\geq 0$.
\end{proposition}
\begin{proof}
The vector bundle and stability statement follow from Propositions \ref{theo4.14} and \ref{prop4.15}. Now we consider the cohomology vanishing condition: by applying the long exact sequence of cohomology to this short exact sequence, one easily concludes that $H^{>0}(E(n))=0$ for $n\geq 0$. Now, by the computation above: $\chi(E(n))=\frac{r}{2}(n-1)^2+(\frac{3}{2}+t)(n-1)+r+t$. Let $n=-1,-3$, $\chi(E(-1))=-t<0$, $\chi(E(-3))=3r-3t>0$, we see from the property of quadratic functions that $\chi(E(n))>0$ for any $n\leq -3$. One can also easily see from the long exact sequence that $H^0(E)=0$ for any $n\leq 0$. Since $\chi(E(n))=h^0(E(n))-h^1(E(n))+h^2(E(n))$, we must have $h^2(E(n))>0$, so $H^2(E(n))\neq 0$ for $n\leq -3$. Finally, by Proposition \ref{prop4.16}, $E$ has natural cohomology, so this forces $H^{<0}(E(n))= 0$ for $n\leq -3$, i.e., $H^{<2}(E(-n-3))=0$ for $n\geq 0$. This completes the proof.
\end{proof}
\begin{corollary}\label{coro4.19}
For any integer $r>0$, there exists an indecomposable maximal Cohen-Macaulay sheaf on $\mathcal{N}_{3,1}$.
\end{corollary}
\begin{proof}
For $r=1,2$, this is the consequence of Proposition \ref{prop3.14}, Corollary \ref{coro4.5}, and Theorem \ref{theo4.14} (for $r=1$, let's state again, $\pi_*f^*\mathcal{O}_{\mathbb{P}^2}(a)$ for $-2\leq a\leq 2$ satisfying the cohomology vanishing condition in Corollary \ref{coro4.5}). For $r\geq3$, by proposition \ref{prop4.17}, it suffices to show there is an integer $t$ with $t< r<\frac{\sqrt5+1}{2}t$. Since $r\geq 3$, let $t=r-1$, we see that $\frac{\sqrt5+1}{2}t-r=\frac{\sqrt5+1}{2}(r-1)-r=\frac{\sqrt5-1}{2}(r-\frac{3+\sqrt{5}}{2})>0$, so we can actually find $t$. The indecomposability follows from Proposition \ref{prop4.6}.
\end{proof}
\section{Constructions on $T^*\mathbb{P}^n\rightarrow \mathcal{N}_{n+1,1}$}\label{s5}
Now we generalize our previous constructions to the resolution $\pi:T^*\mathbb{P}^n\rightarrow \mathcal{N}_{n+1,1}$, where $\mathcal{N}_{n+1,1}$ is the variety of nilpotent $n\times n$ matrices with rank $\leq 1$. Let's consider the analog of Proposition \ref{prop4.4}:
\begin{proposition}
Suppose $f:T^*\mathbb{P}^n\rightarrow\mathbb{P}^n$ is the projection map and $\mathcal{O}_{\mathbb{P}^n}(1)$ is the sheaf of linear forms on $\mathbb{P}^n$. Then, for any vector bundle $E$ on $\mathbb{P}^n$, we have
$$R^{>0}\pi_*f^*E=0 \Longleftrightarrow H^{>0}(E(m))=0\ \ \text{for} \ \ m\geq 0
$$
\end{proposition}
\begin{proof}
The proof is identical to the proof of Proposition \ref{prop4.4} except that this time we should apply the symmetric product to the Euler sequence for $\mathbb{P}^n$.
\end{proof}
Similarly, by Proposition \ref{prop3.14}, we have the following corollary:
\begin{corollary}
Suppose $E$ is a vector bundle on $\mathbb{P}^n$, then 
$$R^{>0}\pi_*f^*E=R^{>0}\pi_*f^*E^\vee=0 \Longleftrightarrow H^{>0}(E(m))=H^{<m}(E(-m-n-1))=0\ \,\text{for}\ m\geq 0
$$
and in this case $\pi_*f^*E$ is maximal Cohen-Macaulay.
\end{corollary}
And similarly, a generalization of Proposition \ref{prop4.6}
\begin{proposition}
Suppose $E$ is a vector bundle on $\mathbb{P}^n$. Then $E$ is indecomposable if and only if $\pi_*f^*E$ cannot be decomposed into the sum of two modules, both of rank $\geq 1$.
\end{proposition}

Now we want to construct indecomposable bundles on $\mathbb{P}^n$  satisfying $H^{>0}(E(m))=H^{<n}(E(-m-n-1))=0$ for $m\geq 0$. The first idea is to check the twists of the tangent bundle and the cotangent bundle of $\mathbb{P}^n$ from our examples $T_{\mathbb{P}^2}(-1)$ and $T_{\mathbb{P}^2}(-2)=\Omega_{\mathbb{P}^2}(-1)$ for $n=2$. To do this, let's compute the cohomology table of $T_{\mathbb{P}^n}$. Twisting the Euler sequence by $m$ and applying the long exact sequence, we get the following table. Here $\star$ means a nonzero number.
$$
\begin{array}{c|cccccccccc}
i \backslash m
& \cdots & -n-3 & -n-2 & -n-1 & -n & \cdots & -2 & -1 & 0 & 1 \ \cdots \\ \hline
h^n
& \cdots &
\star &
\star &
0 & 0 & \cdots & 0 & 0 & 0 & 0\ \cdots \\[6pt]
h^{n-1}
& \cdots & 0 & 0 & \star & 0 & \cdots & 0 & 0 & 0 & 0\ \cdots \\[6pt]
h^{n-2}
& \cdots & 0 & 0 & 0 & 0 & \cdots & 0 & 0 & 0 & 0\ \cdots \\
\vdots
&        & \vdots & \vdots & \vdots & \vdots & & \vdots & \vdots & \vdots & \vdots \\
h^{1}
& \cdots & 0 & 0 & 0 & 0 & \cdots & 0 & 0 & 0 & 0\ \cdots \\[6pt]
h^{0}
& \cdots & 0 & 0 & 0 & 0 & \cdots & 0 & \star & \star &
\star\ \cdots
\end{array}
$$
Then, one easily sees that $T_{\mathbb{P}^n}(j)$ satisfies our cohomology vanishing condition if and only if $-n\leq j\leq -1$. By Serre duality, we see that $\Omega_{\mathbb{P}^n}(l)$, $-n\leq l\leq -1$ are the only twists of the cotangent bundle with the required cohomology vanishing.

To extend the construction in 4.4 to $\mathbb{P}^n$, we need to again consider the quotient $\mathcal{O}_{\mathbb{P}^n}(-1)^{t}\rightarrow\mathcal{O}_{\mathbb{P}^n}^{t+r}$. First, to make it a vector bundle, we have the following proposition (cf.\cite{Harris1992} Proposition 12.2):

\begin{proposition}\label{prop5.4}
Suppose $r\geq n$, a general quotient $\mathcal{O}_{\mathbb{P}^n}(-1)^{t}\rightarrow\mathcal{O}_{\mathbb{P}^n}^{t+r}$ is locally free.
\end{proposition}

To consider stability, we have:
\begin{proposition}[\cite{coskun2025stability}, Theorem 5.1]\label{prop5.5}
If two positive integers $r,t$ satisfy
$$ \frac{n}{t}\leq \frac{r}{t}\leq\frac{n-1+\sqrt{n^2+2n-3}}{2}$$
then a general quotient of $\mathcal{O}_{\mathbb{P}^2}(-1)^{\oplus t}\rightarrow\mathcal{O}_{\mathbb{P}^2}^{\oplus t+r}$ is stable.
\end{proposition}

To consider cohomology vanishing, we need to cite two results in \cite{coskun2025stability}.

\begin{proposition}[\cite{coskun2025stability} Theorem 2.11]\label{prop5.6}
Let $V$ be a bundle of rank $kr$ on $\mathbb{P}^n$ given by a general presentation of the form
$$
0 \rightarrow \mathcal{O}_{\mathbb{P}^n}(-1)^{kt} \rightarrow \mathcal{O}_{\mathbb{P}^n}^{\,k(r+t)} \rightarrow V \rightarrow 0.
$$
Set $\alpha=\left\lceil \frac{tn}{r} \right\rceil$ and $\beta=\left\lfloor \frac{tn}{r} \right\rfloor$. 
If
$$
k \ge\max\left\{\frac{1}{4(t+r)}\binom{\alpha+n-1}{n}^{2}\binom{\alpha+n}{n},\;\frac{1}{4t}\binom{\beta+n-1}{n}^{3}\right\},
$$
then $V$ has natural cohomology.
\end{proposition}

\begin{proposition}[\cite{coskun2025stability} Corollary 4.3]\label{prop5.7}
For integer $q>0$, a general quotient $\mathcal{O}_{\mathbb{P}^n}(-1)^{t}\rightarrow\mathcal{O}_{\mathbb{P}^n}^{t+qn}$ has natural cohomology.
\end{proposition}

Now suppose $(t,r)=(t,qn)$ or $(ka,kb)$ for some $k>>0$, the general quotient of $\mathcal{O}_{\mathbb{P}^n}(-1)^{t}\rightarrow\mathcal{O}_{\mathbb{P}^n}^{t+r}$ has natural cohomology. The question is how we make use of this to produce the required cohomology vanishing. Now apply the long exact sequence of cohomology to $0\rightarrow\mathcal{O}_{\mathbb{P}^n}(m-1)^{t}\rightarrow\mathcal{O}_{\mathbb{P}^n}^{t+r}(m)\rightarrow V(m)\rightarrow 0$, we see that: $H^0(V(m))=0$ for $m\leq -1$; $H^{i}(V(m))=0$ for $1\leq i\leq n-2$; for $i=n-1$ and $n$, by Serre duality, there are
\begin{align}
\notag
H^{n-1}(V(m))&=\ker (H^n(\mathcal{O}_{\mathbb{P}^n}(m-1)^{t})\rightarrow H^n(\mathcal{O}_{\mathbb{P}^n}(m)^{t+r})\\\notag
&=\text{coker}(H^0(\mathcal{O}_{\mathbb{P}^n}(-m-n-1)^{t+r})\rightarrow H^0(\mathcal{O}_{\mathbb{P}^n}(-m-n)^{t}))'
\end{align}
and 
\begin{align}
\notag
H^n(V(m))&=\text{coker} (H^n(\mathcal{O}_{\mathbb{P}^n}(m-1)^{t})\rightarrow H^n(\mathcal{O}_{\mathbb{P}^n}(m)^{t+r})\\\notag
&=\ker(H^0(\mathcal{O}_{\mathbb{P}^n}(-m-n-1)^{t+r})\rightarrow H^0(\mathcal{O}_{\mathbb{P}^n}(-m-n)^{t}))'
\end{align}
Then one can see that $H^{n-1}(V(m))=0$ for $m\geq-n+1$ and $H^{n-1}(V(-n))=\mathbb{C}^t$; $H^n(V(m))=0$ for $m\geq -n$. The cohomology table of $V$ looks like
$$
\begin{array}{c|cccccccccc}
i \backslash m
& \cdots & -n-3 & -n-2 & -n-1 & -n & \cdots & -2 & -1 & 0 & 1 \ \cdots \\ \hline
h^n
& \cdots & \ast & \ast & \ast & 0 & \cdots & 0 & 0 & 0 & 0\ \cdots \\[6pt]
h^{n-1}
& \cdots & \ast & \ast & \ast & t & \cdots & 0 & 0 & 0 & 0\ \cdots \\[6pt]
h^{n-2}
& \cdots & 0 & 0 & 0 & 0 & \cdots & 0 & 0 & 0 & 0\ \cdots \\
\vdots &   & \vdots & \vdots & \vdots & \vdots & & \vdots & \vdots & \vdots & \vdots \\
h^{1}
& \cdots & 0 & 0 & 0 & 0 & \cdots & 0 & 0 & 0 & 0\ \cdots \\[6pt]
h^{0}
& \cdots & 0 & 0 & 0 & 0 & \cdots & 0 & 0 & \ast & \ast\ \cdots
\end{array}
$$

One can see that for $l\geq -n+1$, there is always $H^{>0}(V(l+m))=0$ for $m\geq 0$, so we can consider $E=V(-n+1)$ so that we already have half of the cohomology vanishing $H^{>0}(E(m))=0$ for $m\geq 0$. For the other half, we know that general $V$ has natural cohomology, so in order to have $H^{<n}(E(m))=H^{<n}(V(m-n+1))=0$ for $m\leq -n-1$, it suffices to have $H^n(V(m-n+1))\neq 0$ for $m\leq -n-1$, i.e., $H^n(V(m))\neq 0$ for $m\leq -2n$. By the computations above, as long as
$$h^0(\mathcal{O}_{\mathbb{P}^n}(-m-n-1)^{t+r})> h^0(\mathcal{O}_{\mathbb{P}^n}(-m-n)^{t})
$$
for $m\leq -2n$, so $H^{n}(V(m))=\ker(H^0(\mathcal{O}_{\mathbb{P}^n}(-m-n-1)^{t+r})\rightarrow H^0(\mathcal{O}_{\mathbb{P}^n}(-m-n)^{t}))'$ is nontrivial. Now this is equivalent to 
$$\binom{-m-n-1+n}{-m-n-1}(t+r)>\binom{-m-n+n}{-m-n}t
$$
Simplifying this, one gets $\frac{r}{t}>\frac{n}{-m-n}$. Now $m\leq -2n$, so $-m-n\geq n$. Hence, as long as $\frac{r}{t}>1$, i.e., $r>t$, this would hold, and we have another half of cohomology vanishing.

Combined with Propositions \ref{prop5.4}, \ref{prop5.5}, \ref{prop5.6}, and \ref{prop5.7}, we have the following theorem.
\begin{theorem}\label{theo5.8}
Suppose $r,t$ are integers satisfying $r\geq n$ and $t<r<\frac{n-1+\sqrt{n^2+2n-3}}{2}t$. Let $\phi$ be a general $ k(t+r)\times kt$ matrix of linear forms on $\mathbb{P}^n$ and let $E_{k}$ be the quotient sheaf defined by the sequence
$$0\rightarrow\mathcal{O}_{\mathbb{P}^2}(-n)^{kt}\xrightarrow{\phi}\mathcal{O}_{\mathbb{P}^2}(-n+1)^{k(t+r)}\rightarrow E_k\rightarrow 0
$$
Then $E_k$ is a stable (hence indecomposable) vector bundle, and
\begin{enumerate}
\item There exists an integer $M_{r,t}$ such that for $k\geq M_{r,t}$, we have 
$$H^{>0}(E_k(m))=H^{<n}(E_k(-m-n-1))=0\,\ \text{for}\ m\geq 0$$ 
\item If $r=qn$ for some integer $q$, we can choose $M_{r,t}=1$ in $(1)$.
\end{enumerate}
\end{theorem}
For $n\geq 3$, it's easy to check that for any $r\geq n$, we have $r<\frac{n-1+\sqrt{n^2+2n-3}}{2}(r-1)$, so we can let $t=r-1$. Then Proposition \ref{prop3.14}, Corollary \ref{coro4.5}, and Proposition \ref{prop4.6} give us
\begin{corollary}\label{coro5.9}
Suppose $n\geq 3$. For $r\geq n$, there is an integer $M_{r}$ such that for $k\geq M_{r}$, there exists an indecomposable maximal Cohen-Macaulay module on $\mathcal{N}_{n+1,1}$ of rank $kr$. If $n|r$, then we can set $M_{r}=1$.
\end{corollary}

\ 

\bibliographystyle{amsalpha}
\bibliography{references}

@article{Beauville2000SymplecticSingularities,
  author  = {Beauville, Arnaud},
  title   = {Symplectic singularities},
  journal = {Inventiones Mathematicae},
  volume  = {139},
  number  = {3},
  year    = {2000},
  pages   = {541--549},
  doi     = {10.1007/s002220000088}
}

@article{ItoNakamura1996,
  author  = {Ito, Yukari and Nakamura, Iku},
  title   = {McKay correspondence and Hilbert schemes},
  journal = {Proceedings of the Japan Academy, Series A, Mathematical Sciences},
  volume  = {72},
  number  = {7},
  year    = {1996},
  pages   = {135--138}
}

@article{Nakajima1994Instantons,
  author  = {Nakajima, Hiraku},
  title   = {Instantons on ALE spaces, quiver varieties, and Kac--Moody algebras},
  journal = {Duke Mathematical Journal},
  volume  = {76},
  number  = {2},
  year    = {1994},
  pages   = {365--416},
  doi     = {10.1215/S0012-7094-94-07613-8}
}

@article{CrawleyBoevey2001ExceptionalFibres,
  author  = {Crawley-Boevey, William},
  title   = {On the exceptional fibres of Kleinian singularities},
  journal = {American Journal of Mathematics},
  volume  = {122},
  number  = {5},
  year    = {2000},
  pages   = {1027--1037},
  doi     = {10.1353/ajm.2000.0030}
}

@article{Nakajima1998QuiverKacMoody,
  author  = {Nakajima, Hiraku},
  title   = {Quiver varieties and {K}ac--{M}oody algebras},
  journal = {Duke Mathematical Journal},
  volume  = {91},
  number  = {3},
  year    = {1998},
  pages   = {515--560},
  doi     = {10.1215/S0012-7094-98-09119-7}
}

@incollection{Ginzburg2010Nakajima,
  author    = {Ginzburg, Victor},
  title     = {Lectures on Nakajima's quiver varieties},
  booktitle = {Geometric Methods in Representation Theory II},
  series    = {S\'eminaires et Congr\`es},
  volume    = {24-II},
  pages     = {145--219},
  year      = {2010},
  publisher = {Soci\'et\'e Math\'ematique de France}
}

@misc{Buchweitz1986,
  author       = {Buchweitz, Ragnar-Olaf},
  title        = {Maximal Cohen--Macaulay Modules and Tate Cohomology over Gorenstein Rings},
  year         = {1986},
  note         = {Unpublished manuscript, University of Hannover},
  url          = {https://hdl.handle.net/1807/16682}
}

@book{Hartshorne1977,
  author    = {Hartshorne, Robin},
  title     = {Algebraic Geometry},
  series    = {Graduate Texts in Mathematics},
  volume    = {52},
  publisher = {Springer-Verlag},
  address   = {New York},
  year      = {1977}
}

@book{Eisenbud1995,
  author    = {Eisenbud, David},
  title     = {Commutative Algebra with a View Toward Algebraic Geometry},
  series    = {Graduate Texts in Mathematics},
  volume    = {150},
  publisher = {Springer-Verlag},
  address   = {New York},
  year      = {1995}
}

@article{King1994,
  author    = {King, A.~D.},
  title     = {Moduli of representations of finite dimensional algebras},
  journal   = {The Quarterly Journal of Mathematics},
  volume    = {45},
  number    = {4},
  pages     = {515--530},
  year      = {1994},
  doi       = {10.1093/qmath/45.4.515}
}

@article{ArtinVerdierEsnault1985,
  author    = {Artin, Michael and Verdier, Jean-Louis and Esnault, Hélène},
  title     = {Reflexive modules over rational double points},
  journal   = {Mathematische Annalen},
  volume    = {270},
  number    = {1},
  pages     = {79--82},
  year      = {1985},
  publisher = {Springer}
}

@article{Brieskorn1970,
  author  = {Brieskorn, Egbert},
  title   = {Singular elements of semi-simple algebraic groups},
  journal = {Inventiones Mathematicae},
  volume  = {9},
  year    = {1970},
  pages   = {255--289}
}

@book{ChrissGinzburg1997,
  author    = {Chriss, Neil and Ginzburg, Victor},
  title     = {Representation Theory and Complex Geometry},
  series    = {Progress in Mathematics},
  publisher = {Birkh{\"a}user},
  address   = {Boston},
  year      = {1997}
}

@article{Orlov2004,
  author       = {Orlov, Dmitri},
  title        = {Triangulated categories of singularities and {D}-branes in Landau--Ginzburg models},
  journal      = {Proceedings of the Steklov Institute of Mathematics},
  volume       = {246},
  year         = {2004},
  pages        = {227--248},
  arxiv        = {math/0302304}
}

@article{Lu2020TriangleEquivalences,
  author  = {Lu, Li},
  title   = {Triangle equivalences and Gorenstein schemes},
  journal = {International Journal of Mathematics and Computer Science},
  volume  = {15},
  number  = {1},
  year    = {2020},
  pages   = {301--307}
}

@book{Weibel1994,
  author    = {Weibel, Charles A.},
  title     = {An Introduction to Homological Algebra},
  series    = {Cambridge Studies in Advanced Mathematics},
  volume    = {38},
  publisher = {Cambridge University Press},
  address   = {Cambridge},
  year      = {1994},
  isbn      = {978-0-521-43500-6}
}

@book{HartshorneRD,
  author    = {Robin Hartshorne},
  title     = {Residues and Duality},
  series    = {Lecture Notes in Mathematics},
  volume    = {20},
  publisher = {Springer-Verlag},
  address   = {Berlin},
  year      = {1966}
}

@book{OSS1980,
  author    = {Okonek, Christian and Schneider, Michael and Spindler, Heinz},
  title     = {Vector Bundles on Complex Projective Spaces},
  series    = {Progress in Mathematics},
  volume    = {3},
  publisher = {Birkh\"auser},
  address   = {Boston},
  year      = {1980}
}

@book{EisenbudHarris2016,
  author    = {Eisenbud, David and Harris, Joe},
  title     = {3264 and All That: A Second Course in Algebraic Geometry},
  series    = {Cambridge University Press},
  publisher = {Cambridge University Press},
  address   = {Cambridge},
  year      = {2016}
}

@article{Hoppe1984,
  author = {Hoppe, Hans},
  title = {Generischer Spaltungstyp und zweite Chernklasse stabiler Vektorraumb{\"u}ndel vom Rang 4 auf $\mathbb{P}^4$},
  journal = {Mathematische Zeitschrift},
  volume = {187},
  year = {1984},
  pages = {345--360}
}

@book{EisenbudSyzygies,
  author    = {Eisenbud, David},
  title     = {The Geometry of Syzygies: A Second Course in Commutative Algebra and Algebraic Geometry},
  series    = {Graduate Texts in Mathematics},
  volume    = {229},
  publisher = {Springer},
  year      = {2005}
}

@book{Harris1992,
  author = {Joe Harris},
  title = {Algebraic Geometry: A First Course},
  series = {Graduate Texts in Mathematics},
  volume = {133},
  publisher = {Springer},
  address = {New York},
  year = {1992}
}

@article{coskun2025stability,
  title={Stability and cohomology of kernel bundles on Pn},
  author={Coskun, Izzet and Huizenga, Jack and Smith, Geoffrey},
  journal={Michigan Mathematical Journal},
  volume={75},
  number={1},
  pages={173--198},
  year={2025},
  publisher={University of Michigan, Department of Mathematics}
}

@article{Leuschke2007,
  author  = {Graham J. Leuschke},
  title   = {Endomorphism Rings of Finite Global Dimension},
  journal = {Canadian Journal of Mathematics},
  volume  = {59},
  number  = {2},
  pages   = {332--350},
  year    = {2007}
}

@article{IyamaWemyss2014,
  author  = {Osamu Iyama and Michael Wemyss},
  title   = {Maximal Modifications and Auslander--Reiten Duality for Non-isolated Singularities},
  journal = {Inventiones Mathematicae},
  volume  = {197},
  number  = {3},
  pages   = {521--586},
  year    = {2014}
}

@book{deBobadillaRomano2024,
  author    = {Javier Fern{\'a}ndez de Bobadilla and Agust{\'\i}n Romano-Vel{\'a}zquez},
  title     = {Reflexive Modules on Normal Gorenstein Stein Surfaces, Their Deformations and Moduli},
  series    = {Memoirs of the American Mathematical Society},
  volume    = {298},
  number    = {1493},
  publisher = {American Mathematical Society},
  year      = {2024},
  pages     = {94}
}

@article{EsnaultKnorrer1985,
  author  = {H{\'e}l{\`e}ne Esnault and Horst Kn{\"o}rrer},
  title   = {Reflexive Modules over Rational Double Points},
  journal = {Mathematische Annalen},
  volume  = {272},
  number  = {4},
  pages   = {545--548},
  year    = {1985},
  doi     = {10.1007/BF01455865}
}

@article{Ishii1992,
  author  = {Akira Ishii},
  title   = {On the Moduli of Reflexive Sheaves on a Surface with Rational Double Points},
  journal = {Mathematische Annalen},
  volume  = {294},
  number  = {1},
  pages   = {125--150},
  year    = {1992},
  doi     = {10.1007/BF01934318}
}
\end{document}